\documentclass{amsart}
\usepackage{comment}
\usepackage{amscd}
\usepackage{epsf,latexsym,graphicx,dsfont}
\usepackage[matrix,arrow,tips,curve]{xy}
\usepackage{color}
\usepackage{tikz}
\usepackage{euscript}
\newlength{\defbaselineskip} 
\usepackage{footnote}
\usepackage{color}
\usepackage{todonotes}
\usepackage{MnSymbol}

\usepackage{amsfonts}
\usepackage[centertags]{amsmath}
\makeatletter
\newcommand{\xleftrightarrow}[2][]{\ext@arrow 3359\leftrightarrowfill@{#1}{#2}}
\newcommand{\xdashrightarrow}[2][]{\ext@arrow 0359\rightarrowfill@@{#1}{#2}}
\newcommand{\xdashleftarrow}[2][]{\ext@arrow 3095\leftarrowfill@@{#1}{#2}}
\newcommand{\xdashleftrightarrow}[2][]{\ext@arrow 3359\leftrightarrowfill@@{#1}{#2}}
\def\rightarrowfill@@{\arrowfill@@\relax\relbar\rightarrow}
\def\leftarrowfill@@{\arrowfill@@\leftarrow\relbar\relax}
\def\leftrightarrowfill@@{\arrowfill@@\leftarrow\relbar\rightarrow}
\def\arrowfill@@#1#2#3#4{%
  $\m@th\thickmuskip0mu\medmuskip\thickmuskip\thinmuskip\thickmuskip
   \relax#4#1
   \xleaders\hbox{$#4#2$}\hfill
   #3$%
}
\makeatother

\newtheorem{thm}{Theorem}[section]
\newtheorem{cor}[thm]{Corollary}

\newtheorem{lemma}[thm]{Lemma}
\newtheorem{lem}[thm]{Lemma}
\newtheorem{prop}[thm]{Proposition}

\theoremstyle{definition}

\newtheorem{rem}[thm]{Remark}

\theoremstyle{remark}
\newtheorem*{proof1}{Proof of Theorem \ref{thm:main}}
\newtheorem*{proof2}{Proof of Corollary \ref{cor:stability}}
\usepackage{fullpage}
\usepackage{tikz,tikz-cd}
\usetikzlibrary{matrix,arrows}
\makeatletter
\tikzset{
  edge node/.code={%
      \expandafter\def\expandafter\tikz@tonodes\expandafter{\tikz@tonodes #1}}}
\makeatother
\tikzset{
  subseteq/.style={
    draw=none,
    edge node={node [sloped, allow upside down, auto=false]{$\subseteq$}}},
  Subseteq/.style={
    draw=none,
    every to/.append style={
      edge node={node [sloped, allow upside down, auto=false]{$\subseteq$}}}
  }
}

 \numberwithin{equation}{section}
\numberwithin{equation}{section} \theoremstyle{definition}

\DeclareMathOperator{\Hom}{Hom}

\DeclareMathOperator{\Sym}{Sym}

\DeclareMathOperator{\sym}{Sym}

          \newcommand\PP{{\mathbb{P}}}
          \newcommand\ZZ{{\mathbb{Z}}}

          \newcommand\oo{\mathcal O}
          \newcommand\Z{\mathbb{Z}}

\definecolor{zielony}{rgb}{0.5, 0.9, 0.1}
\definecolor{czerwony}{rgb}{0.8, 0.2, 0.1}
\definecolor{niebieski}{rgb}{0.3, 0.1, 0.9}

\newcounter{appendice}

\author[Grzegorz Kapustka]{Grzegorz Kapustka}
\address{G. Kapustka: Department of Mathematics and Informatics, Jagiellonian University, \L ojasiewicza 6, 30-348, Krak\'ow, Poland}
\email{grzegorz.kapustka@uj.edu.pl}
\author[Micha\l \ Kapustka]{Michal Kapustka}
\address{M. Kapustka: Institute of Mathematics of the Polish Academy of Sciences, ul. Śniadeckich 8, 00-656 Warszawa, Poland}
\email{michal.kapustka@impan.pl}
\author[Giovanni Mongardi]{Giovanni Mongardi}
\address{G. Mongardi: Department of Mathematics, Alma mater studiorum Università di Bologna, P.zza di porta san Donato 5, 40126 Bologna, Italy}
\email{giovanni.mongardi2@unibo.it}
\begin{document}
\title{EPW sextics vs EPW cubes}
\maketitle
\begin{abstract}
      
      
      We study a correspondence between double EPW cubes and double EPW sextics, two families of polarized hyper-K\"ahler manifolds related to Gushel--Mukai fourfolds. We infer relations between these families in terms of Hodge structures and moduli spaces of elliptic curves.
      As an application, we prove that a very general double EPW cube is the moduli space of stable objects with respect to a suitable stability condition on the Kuznetsov component of its corresponding Gushel--Mukai fourfolds; this answers a problem posed by Perry, Pertusi and Zhao.

\end{abstract}
\section{Introduction}

Let $W$ be a $6$-dimensional complex vector space. Let $$\textstyle LG(10,\bigwedge^3 W ) \subset G(10, \bigwedge^3 W )$$ be the Lagrangian Grassmannian corresponding to the conformal symplectic structure on $\bigwedge^3 W$ given by wedge product. 
For a fixed $A\in LG(10,\bigwedge^3 W)$,
 there are two naturally associated degeneracy loci: 
\begin{equation}\textstyle
	X_A:=\{[v]\in \mathbb P(W)\mid\ \dim (v\wedge \bigwedge^2 W)\cap A\geq 1 \}
 \end{equation}
 is the EPW sextic associated to $A$,
	
	\begin{equation} \textstyle Y_A:=\{[U]\in G(3,W)\mid\ \dim (\bigwedge^2 U\wedge W)\cap A\geq 2 \}
 \end{equation}
 is the EPW cube associated to $A$.
\bigskip \\
When $A$ is general, canonical double covers of $X_A$ and $Y_A$ were constructed; they are polarized hyper-K\"{a}hler manifolds of $K3^{[n]}$-type: the double EPW sextic $\widetilde{X}_A$ of dimension 4, and the double EPW cube $\widetilde{Y}_A$ of dimension 6 \cite{O2,IKKR1}. 

Recall that for any hyper-K\"{a}hler manifold $X$ there is a natural integral nondegenerate quadratic form on $H^2(X,\mathbb Z)$ called the Beauville-Bogomolov-Fujiki (BBF) form. The BBF is used to classify families of polarised hyper-K\"{a}hler manifolds. In this case, double EPW sextics have polarisation of BBF degree 2, whereas double EPW cubes have polarisation of degree 4 and divisibiility 2 with respect to the BBF.

In this work we study the relations between the manifolds $\widetilde{X}_A$, $\widetilde{Y}_A$, associated to the same Lagrangian $A$. In particular, it appears that although $\widetilde{X}_A$ and $\widetilde{Y}_A$ have different dimension they are deeply connected.

 The main object relating the two families is obtained as a correspondence. More precisely, 
 to
 a Lagrangian $A\subset \bigwedge^3 W$ we associate  a further variety, the incidence correspondence $$K_A=\{([U],[\alpha])\in Y_A\times \mathbb{P}(\textstyle \bigwedge^3 W)\mid\ [\alpha] \in \PP(T_U)\cap \PP(A)\cap \Omega\},$$
 where $\Omega\subset \PP( \bigwedge^3 W)$ is the subset of  $3$-vectors decomposing into a product of a $1$-vector and a $2$-vector, and $T_U:=\bigwedge^2 U\wedge W$.
 The sixfold $K_A$ admits a natural generically $3:1$ map to $Y_A$
and a map to $X_A$ which is generically a fibration by surfaces of general type with even sets of nodes. 

 Recall that the primitive part $H^2_{{\operatorname{prim}}}(X,\mathbb{Z})$ of the Hodge structure of a polarized hyper-K\"ahler manifold $(X, L)$ is the induced Hodge structure on the  orthogonal complement in $H^2(X,\mathbb{Z})$ with respect to the BBF form, of the class in $H^2(X,\mathbb{Z})$ of the polarization $L$. The main result of the paper is that the primitive parts of the Hodge structures of double EPW sextics $\widetilde{X}_A$ and cubes $\widetilde{Y}_A$ only depend on the Lagrangian $A$.

\begin{thm}\label{thm:main}
Let $A$ be a general Lagrangian subspace of $\bigwedge^3 W$ and let $\widetilde{X}_A$ and $\widetilde{Y}_A$ be the double EPW sextic and cube associated to $A$. Then the polarized integral Hodge structures $H^2_{{\operatorname{prim}}}(\widetilde{X}_A,\mathbb{Z})$ and $H^2_{{\operatorname{prim}}}(\widetilde{Y}_A,\mathbb{Z})$ are isometric.
\end{thm}
\noindent Here, general means a Lagrangian as in \eqref{eq:Lagrangian_condition}, see Section \ref{sec:prel} for a precise explanation of the conditions.

Another intriguing consequence of the investigation of the correspondence $K_A$ is a relation between  moduli spaces of minimal degree elliptic curves on a double EPW sextic $\widetilde{X}_A$ and a double EPW cube $\widetilde{Y}_A$ associated to the same Lagrangian $A$. More precisely, for a hyper-K\"ahler manifold $\widetilde{X}$ and chosen curve cohomology class $\beta$,
 let $M_{1,0}(\widetilde{X},\beta)$ denote the moduli space (not necessarily compact) of maps $f\,\colon\,  C \to \widetilde{X}$ from smooth curves $C$ of genus $1$ such that $f_*[C]=\beta$.
Note that if the hyper-K\"ahler manifold $\widetilde X$ has Picard number $1$, there exists a unique curve class of minimal degree. If $\beta$ denotes this class then $M_{1,0}(\widetilde{X},\beta)$ is called the moduli space of minimal degree elliptic curves. 

 It can be shown that the expected dimension of this moduli space is always $1$, independently of the dimension of $\widetilde{X}$. It was conjectured furthermore in \cite[Conjecture 1.1]{NOb} that this moduli space is indeed always of dimension $1$. As evidence, it was proven in \cite[Theorem 1.3]{NOb} that the moduli space of elliptic curves of minimal degree in the Fano variety of  a general cubic fourfold is a curve. We extend this result proving  \cite[Conjecture 1.1]{NOb} for very general double EPW sextics and double EPW cubes and compare the two obtained curves.
  
\begin{thm}\label{moduli elliptic} Let $A$ be a very general Lagrangian subspace of $\bigwedge^3 W$ and let $\widetilde{X}_A$ and $\widetilde{Y}_A$ be the double EPW sextic and cube associated to $A$. Let $\beta_X$ and $\beta_Y$ be the cohomology classes of curves of minimal degree on $\widetilde{X}_A$ and $\widetilde{Y}_A$ respectively. 
 Then the moduli spaces $M_{1,0}(\widetilde{X}_A,\beta_X)$ and  $M_{1,0}(\widetilde{Y}_A,\beta_Y)$ are $1$-dimensional and their reduced structures are isomorphic to each other.
\end{thm}
 Theorem \ref{moduli elliptic} provides another interesting relation between double EPW sextics and double EPW cubes associated to the same Lagrangian that seems not to be directly related with the fact that the EPW sextic and double EPW cube are moduli spaces of stable objects on the same K3 category.

 The isometry of Theorem \ref{thm:main} is finally used to describe very general EPW cubes as moduli spaces of stable objects on some K3 category.
 For that, recall first that Gushel--Mukai (GM) fourfolds are smooth transversal intersections $$Z =CG(2,5)\cap H_1\cap H_2\cap Q\subset \PP^{10},$$ where $H_1,H_2$ are hyperplanes, $Q$ is a quadric and $CG(2,5)$ is the cone over the Pl\"ucker embedding of the Grassmanian of planes in a $5$-dimensional complex vector space. They are expected to behave in a similar way to cubic fourfolds \cite{DK,DK moduli,KP1}. They also have associated hyper-K\"ahler manifolds and K3 categories \cite{BLMNPS,D}, and their rationality is conjectured in terms of these invariants, see \cite{D2} for an excellent survey on these manifolds. 

 We prove that double EPW cubes are moduli spaces of stable objects on the Kuznetsov components of some GM fourfolds, with respect to some stability conditions and Mukai vectors (see \cite[Part VI]{BLMNPS} and also \cite{PPZ}). To be more precise, recall from \cite{DK}, that to a Lagrangian subspace $A\subset \bigwedge^3 W$ one can associate a $4$-dimensional family of birational GM fourfolds, which are called period partners. Furthermore, by \cite[Theorem 4.5]{D2}, it is known that  $$H^2_{{\operatorname{prim}}}(\widetilde{X}_A,\mathbb{Z})\cong H^4_{{\operatorname{van}}}(Z_A,\mathbb{Z})(-1),$$ for any GM fourfold $Z_A$ associated to $A$. Here $H^4_{{\operatorname{van}}}(Z_A,\mathbb{Z})(-1)$ stands for the Hodge structure on the vanishing cohomology of $Z_A$ i.e. on the orthogonal complement with respect to cup product of the restriction of $H^4(G(2,5), \mathbb Z))$ to $H^4(Z_A, \mathbb Z)$. It follows that also $$H^2_{{\operatorname{prim}}}(\widetilde{Y}_A,\mathbb{Z})\cong H^4_{{\operatorname{van}}}(Z_A,\mathbb{Z})(-1).$$ Recall also that for any GM fourfold $Z_A$ the Kuznetsov component $\operatorname{Ku}(Z_A)$ of the bounded derived category of coherent sheaves $ D^b(Z_A)$ on $Z_A$ is defined by the semiorthogonal decomposition
$$D^b(Z_A) = \langle \operatorname{Ku}(Z_A), \mathcal{O}_{Z_A} , U_{Z_A}^{\vee}, \mathcal{O}_{Z_A} (1), U_{Z_A}^{\vee}(1)\rangle, $$
where $U_{Z_A}$ and $\mathcal{O} _{Z_A}(1)$ denote the pullbacks to $Z_A$ of the rank 2 tautological subbundle and Pl\"ucker line bundle on $G(2,5)$, see \cite[Proposition 2.3]{KP1}. The Kuznetsov component $\operatorname{Ku}(Z_A)$ of a $GM$ fourfold is a K3 category and depends only on the Lagrangian subspace $A$ as proven in \cite[Theorem 1.6]{KP2}. We denote it by $\operatorname{Ku}(A)$.
It is known \cite[Proposition~5.17]{PPZ} that for a very general choice of $A$ the double EPW sextic $\widetilde{X}_A$ is a moduli space of stable objects on $\operatorname{Ku}(A)$. As a corollary of Theorem \ref{thm:main}, we obtain in Section \ref{sec:hodge}  a similar statement about EPW cubes conjectured in \cite[Section 5.4.2]{PPZ}. 
 
 \begin{cor}\label{cor:stability}
For a very general Lagrangian $A\subset \bigwedge^3 W$, there exist a stability condition $\sigma$ on $\operatorname{Ku}(A)$ and a Mukai vector $v$ such that the moduli space $M_{\sigma}(\operatorname{Ku}(A),v)$ is a hyper-K\"ahler sixfold isomorphic to the double EPW cube $\widetilde{Y}_A$ corresponding to $A$.
 \end{cor}
Additionally, we analyze in Section \ref{s7} the impact of Theorem \ref{thm:main} on the description of the period domain $\mathbb{D}_Y$ of double EPW cubes and the  period map $$\mathcal{P}_Y\,\colon\,\, LG(10,\textstyle \bigwedge^3W) \dashrightarrow \mathbb{D}_{Y},$$
that associates to a general Lagrangian space the period of the corresponding double EPW cube. We note that the period map is well defined on the locus $LG(10,\bigwedge^3 W)\setminus \Sigma$, where $\Sigma$ is the locus of Lagrangian spaces that contain some totally decomposable 3-vectors.
 We start by relating the period domains of double EPW sextics and cubes, and use this to obtain a description, analogous to \cite{O3}, 
 of the hyper-K\"ahler sixfolds represented by points in $\mathbb D_Y$ for which are in the complement of the image of 
 $LG(10,\bigwedge^3 W)\setminus \Sigma$ by the above period map. In particular, we describe geometrically the corresponding degenerate EPW cubes.

The paper is organized as follows: in Section 2, we introduce notation used throughout the paper and fix a Lagrangian space $A$. Section 3 is devoted to the introduction of the variety $K_A$, the study of its maps to $X_A$ and $Y_A$ and of the correspondence $k_A$ between the cohomology of $\widetilde X_A$ and $\widetilde Y_A$ that they induce.
In Section 4, we use $K_A$ to describe and compare the reduced structures of moduli spaces of elliptic curves of minimal degree on $\widetilde X_A$ and $\widetilde Y_A$, thus proving Theorem \ref{moduli elliptic}. In Section 5 we prove Theorem \ref{thm:main} and Corollary \ref{cor:stability} using the variety $K_A$. In the last section, we deduce consequences of Theorem \ref{thm:main} for the period maps and describe the geometry of boundary cases. Finally, the Appendix contains a Macaulay2 script that is used in the proofs in Section 4.

\section*{Acknowledgements}
G.K.~is supported by the project Narodowe Centrum Nauki 2018/30/E/ST1/00530. M.K.~is supported by the project Narodowe Centrum Nauki 2018/31/B/ST1/02857. G.M.~is supported by the projects PRIN 2020KKWT53, PRIN 2022PEKYBJ and is a member of INdAM - GNSAGA and received support from it. The authors would like to thank Olivier Debarre, Atanas Iliev, Kieran G. O'Grady, Luca Migliorini, and Kristian Ranestad for useful discussions and comments, and Marcello Bernardara, Enrico Fatighenti, Laurent Manivel and Fabio Tanturri for letting us know about \cite{B+} and for useful discussions. We would also like to thank the referee for helpful comments and suggestions.
Funded by the European Union - NextGenerationEU under the National Recovery and Resilience Plan (PNRR) - Mission 4 Education and research - Component 2 From research to business - Investment 1.1 Notice Prin 2022 - DD N. 104 del 2/2/2022, from title "Symplectic varieties: their interplay with Fano manifolds and derived categories", proposal code 2022PEKYBJ – CUP J53D23003840006.
\section {Preliminaries}\label{sec:prel}

Let $W$ be a $6$-dimensional vector space and let $G(3,W)\subset \PP(\bigwedge^3 W)$ be the Grassmannian of 3-spaces in $W$. 
Let $A\subset \bigwedge^3 W$ be a Lagrangian subspace with respect to the conformal symplectic structure on $\bigwedge^3 W$ given by the wedge product. There are three hyper-K\"ahler manifolds associated to a general $A$, and all of them arise as the double covers of Lagrangian degeneracy loci of Lagrangian bundles.
Set
\begin{itemize}
\item $F_v=v\wedge \bigwedge^2 W \subset  \bigwedge^3 W$ for $v\in W$,
\item $F'_H= (H\wedge \bigwedge^2 W^{\vee})^{\perp}\subset \bigwedge^3 W$ for $H\in W^{\vee}$,
\item $T_U=\bigwedge^2 U\wedge W$ for $[U]\in G(3,W)$.
\end{itemize}
These represent fibers of Lagrangian subbundles $\mathcal F$, $\mathcal F'$, $\mathcal T$ of the trivial bundle with fiber $\bigwedge^3 W$ over $\mathbb{P}(W)$, $\mathbb{P}(W^{\vee})$ and $G(3,W)$ respectively. Note that $\PP(F_v)\cap G(3,W)$ is $G(2,5)$ embedded in $\mathbb{P}(F_v)$ via the Pl\"ucker embedding, whereas $\mathbb{P}(T_U) \cap G(3,W)$ is a cone $C_U$ over $\mathbb P^2\times \mathbb P^2$, which is the Schubert cycle in $G(3,W)$ of 3-spaces meeting $U$ in a $2$-dimensional space.  

We have the following degeneracy loci
\begin{enumerate}
\item $D^i_{A,\mathcal F} :=\{[v]\in \mathbb P(W)\mid\ \dim (F_v\cap A)\geq i \}$,
\item  $D^i_{A,\mathcal F'} :=\{[H]\in \mathbb P(W^{\vee})\mid\ \dim( F'_H\cap A)\geq i \}$.
\item  $D^i_{A,\mathcal T} :=\{[U]\in  G(3,W)\mid\ \dim( T_U\cap A)\geq i \}$. 
\end{enumerate}
We call $X_A:=D^1_{A,\mathcal F}$ the EPW sextic associated to $A$, $X_{A^{\vee}}:=D^1_{A,\mathcal F'}$ the dual EPW sextic associated to $A$ and
$Y_A:=D^2_{A,\mathcal T}$
the EPW cube associated to $A$.
 
From now on, we make the following assumption on $A$:
\begin{equation}\label{eq:Lagrangian_condition}
    D^3_{A,\mathcal F}=D^3_{A,\mathcal F'}=D^4_{A,\mathcal T}=\mathbb{P}(A)\cap G(3,W)=\emptyset.
\end{equation}
Notice that each condition fails in codimension one by \cite[Sections 1.4 and 1.5]{O3} and \cite[Lemma 3.6]{IKKR1}, so the set of Lagrangians $A$ that we are considering is an open and dense subset of the Lagrangian Grassmannian $LG(10,\bigwedge^3 W)$. We will follow the standard notations for these Lagrangians by denoting with $\Sigma$ the set of Lagrangians $A$ such that $\mathbb{P}(A)\cap G(3,W)\neq\emptyset$, $\Delta$ the set of Lagrangians $A$ such that $D^3_{A,\mathcal F}\neq\emptyset$ and $\Gamma$ the set of Lagrangians $A$ such that $D^4_{A,\mathcal{T}}\neq\emptyset$. 

By \cite[Theorem 1.1]{IKKR1} and \cite[Theorem 1.1]{O2}, there are canonical double covers of $X_A$, $X_{A^{\vee}}$ and $Y_A$ that are smooth hyper-K\"ahler manifolds: they are called respectively double EPW sextics and cubes, and we denote them by $\widetilde{X}_A$, $\widetilde{X}_{A^{\vee}}$ and $\widetilde{Y}_A$ respectively. The covering maps are denoted by $$p_X\,\colon\,\widetilde{X}_A \to X_A,$$ $$p_{X^{\vee}}\,\colon\,\widetilde{X}_{A^{\vee}} \to X_{A^{\vee}},$$ $$p_Y\,\colon\,\widetilde{Y}_A \to Y_A,$$ and the three covering involutions will be called $\iota_X,$ $\iota_{X^{\vee}}$ and $\iota_Y$ respectively. 

\section{The incidence correspondence}\label{sec:correspondence}
Let us denote by
$$ \Omega=\bigcup_{v\in W} \PP(F_v)=\bigcup_{H\in W^{\vee}} \PP(F'_H)
\textstyle
=\{[v\wedge \alpha] \mid  v\in W,  \alpha\in \bigwedge^2 W\}\subset \mathbb{P}(\bigwedge^3 W),$$ the $14$-dimensional variety which is the closure of the middle orbit of the representation $\mathbb{P}(\bigwedge^3 W)$ of $GL(W)$. Note $\Omega\setminus G(3,W)$ admits two fibrations by $\PP(F_v)\setminus G(3,W)$ with $v\in W$ and $\PP(F'_H)\setminus G(3,W)$ with $H\in W^{\vee}$. More precisely, for all $[\beta] \in \Omega\setminus G(3,W)$ there is exactly one $[v]\in \PP (W)$ and $[H]\in \PP( W^{\vee})$ such that $[\beta] \in \PP(F_v)$ and $[\beta] \in \PP(F'_H)$. In fact, for $\beta=v\wedge \alpha$ we have $[\beta]\in \mathbb P(F_{v})$ and $[\beta] \in \PP(F'_H)$ for $[H] =[v\wedge \alpha\wedge \alpha]\in \PP(\bigwedge^5 W)=\PP(W^{\vee})$.
Summing up, we have two fibrations:
\begin{align*}
    \phi_1\,\colon\, \Omega \setminus G(3,W)&\longrightarrow \mathbb{P}(W)\\ [v\wedge \alpha] &\longmapsto [v]
\end{align*} and 
\begin{align*}\phi_2\,\colon\, \Omega\setminus G(3,W)  &\longrightarrow \mathbb P(\textstyle \bigwedge^5 W)=\mathbb{P}(W^{\vee})\\
 [v\wedge \alpha] &\longmapsto [v\wedge \alpha\wedge \alpha]
\end{align*}
whose fibers are $\mathbb{P}^9\setminus G(2,5)$.
\\
We can then consider the variety:
$$K_A=\{([U],[\alpha])\in  Y_A\times  \mathbb{P}(\textstyle\bigwedge^3 W) \mid\ [\alpha] \in \PP(T_U)\cap \PP(A)\cap \Omega\}.$$
Note that $K_A$ admits the following three natural maps: the projection $$\pi_{1}\,\colon\, K_A\to Y_A\subset G(3,W),$$ and two maps $\psi_{1}$, $\psi_{2}$ obtained as compositions of the projection 
$$\pi_{2}\,\colon\, K_A\to \mathbb{P}(\textstyle\bigwedge^3 W),$$ whose image is contained in $\Omega\cap \PP(A)\subset \Omega\setminus G(3,W)$ (by assumption \eqref{eq:Lagrangian_condition}), with the fibrations $\phi_1$ and $\phi_2$.  

\begin{lem} \label{fibers pi} Let $A$ be a general Lagrangian subspace of $\bigwedge^3 W$. The variety $K_A$ is then  $6$-dimensional and the map $\pi_1$ is generically 3 to 1 onto $Y_A$. More precisely, for general $[U]\in Y_A$, the intersection $l_{U,A}:=\mathbb{P}(A)\cap \mathbb{P}(T_U) $ is a line and $R_U:=\mathbb{P}(T_U)\cap \Omega$ is a determinantal cubic hypersurface cone in $\mathbb{P}(T_U)$, which is the secant variety of the cone $C_U=\mathbb P(T_U)\cap G(3,W)$. In this case, the fiber $\pi_1^{-1}([U])$  consists of three points corresponding to three points of intersection of the line $l_{U,A}$ with the determinantal cubic $R_U$. 
\end{lem}
\begin{proof} 
Note that the dimension of $K_A$ follows from the description of the general fibers of $\pi_1$ once we observe that the fibers of $\pi_1$ are at most $1$-dimensional on the smooth locus of $Y_A=D^2_{A,\mathcal T}$ and at most $2$-dimensional over its singular locus $D^3_{A,\mathcal T}$, which is a threefold.
Recall from \cite[Lemma 3.5]{Don} that $C_U=G(3,W)\,\cap\, \mathbb{P}(T_U)$ is a cone over the Segre embedding of $\mathbb{P}^2\times \mathbb P^2$. Let us decompose $W=U\oplus V$ then $GL(U)\times GL(V)$ acts naturally on $\mathbb P(\bigwedge^3 W)=\mathbb P(\bigwedge^3 U \oplus \bigwedge^2 U \wedge V \oplus U \wedge \bigwedge^2 V \oplus \bigwedge^3 V)$ and preserves $\mathbb{P}(T_U)=\mathbb P(\bigwedge^3 U \oplus \bigwedge^2 U \wedge V)$ as well as all orbits of the action of $GL(W)$. In particular, $\Omega \cap \mathbb{P}(T_U)$ and $G(3,W)\cap \mathbb{P}(T_U)$ are preserved by this action. Furthermore, $[U]=\mathbb P(\bigwedge^3 U)$ is fixed by the action. Note also that the representation $\mathbb P(\bigwedge^2 U \wedge V)$ of $GL(U)\times GL(V)$ is isomorphic to $\mathbb P (\Hom(U,V))$. The latter is a standard representation of $GL(3)\times GL(3)$ with exactly three orbits determined by the rank loci of the maps in $\Hom(U,V)$. Under these identifications $C_U$ is identified with the cone over the locus of rank one maps  in $\mathbb P (\Hom(U,V))$. 
The secant variety of $C_U$ is then identified with the cone over the locus of equivalence classes of  rank two maps  in $\mathbb P (\Hom(U,V))$ i.e. is a cone over a determinantal cubic hypersurface.
Now, $R_U$ is also the closure of an orbit of the action. To prove that it is indeed the secant variety we just need to observe that it contains the secant variety of $C_U$ and is not the whole $\mathbb P(T_U)$.
For that note that a general $3$-vector corresponding to a point in the secant variety to $C_U$ can be written as $v_1\wedge v_2\wedge w_1+v_2\wedge v_3\wedge w_2=v_2\wedge (v_3\wedge w_2-v_1\wedge w_1)$ with $v_1,v_2,v_3\in U$,$w_1,w_2,w_3\in V$ and thus is a point in $\Omega$. On the other hand a $3$-vector representing a general point in $\mathbb{P}(T_U)$ is of the form $$v_1\wedge v_2\wedge w_1+v_2\wedge v_3\wedge w_2+v_1\wedge v_3\wedge w_3,$$ with $v_1,v_2,v_3\in U$,$w_1,w_2,w_3\in V$ forming together a basis of $W$. The latter is a $3$-vector whose class is not in $ \Omega$. This determines $R_U$ as the secant variety to $C_U$. It remains to observe that since $A$ is general we can assume $l_{U,A}$ meets $R_U$ transversely for some and hence for general $[U]\in Y_A$, which ends the proof. 
\end{proof}


\begin{lem}\label{fibers} Let $A$ be a general Lagrangian subspace of $\bigwedge^3 W$. The map $\psi_1$ (resp. $\psi_2$) is then a fibration whose image is $X_A $ (resp.~$X_{A^{\vee}} $) and whose general fibers are surfaces of general type, with even sets of 20 nodes. Moreover, each such surface is isomorphic to an intersection of a quadric and a quartic in $\mathbb{P}^4$.
\end{lem}
\begin{proof} Let $p\in K_A$. Then by definition $\pi_2(p)\in \Omega \cap \mathbb{P}(A) $, hence $\pi_2(p)\in \mathbb{P}(A)\cap \mathbb{P}(F_{\phi_1(\pi_2(p))})$. It follows that $\psi_1(p)\in X_A$. 

Consider now a smooth point $v\in X_A$, which under our assumption on $A$ means in particular that $v\in D^1_{A,F}\setminus D^2_{A,F}$. Then there is a unique point of intersection 
$p:=\mathbb{P}(A)\cap\mathbb{P}(F_v)$. Now, the fiber over $v$ is $$\psi_1^{-1}(v)=\{([U], p)\in Y_A\times \{p\} \mid \ p\,\in \mathbb{P}(T_U)\}.$$ 
Since $p$ is uniquely determined by $v$, the fiber $\psi_1^{-1}(v)$ is isomorphic to its projection $\pi_1(\psi_1^{-1}(v))\subset Y_A\subset G(3,W)$.  To describe this image, let  us first describe 
$$Q_p:=\{[U]\in G(3,W) \mid \ p\in \mathbb P (T_U)\}.$$
Since $p\in \mathbb P(F_v)\setminus G(3,W) \subset \Omega\setminus G(3,W)$, we can write $p=[v\wedge \alpha]$ with $\alpha\in \bigwedge^2 W$ non-degenerate. Observe that 
\begin{equation}\textstyle \label{pinTU}[v\wedge \alpha]\in \mathbb{P}(T_U)\iff  v\in U \text{ and }\alpha\wedge \bigwedge^3 U=0.\end{equation}
Now, if we denote by $V_{v,\alpha}$ the 5-space corresponding to the $5$-vector $v\wedge \alpha\wedge\alpha$, then for a fixed $p$, the condition (\ref{pinTU}) describes $Q_p$ as a $3$-dimensional quadric obtained as 
a hyperplane section of the variety $G(v, 3, V_{v,\alpha} )\cong G(2,4)$ of 3-spaces containing $v$ and contained in $V_{v,\alpha}$. 

Now, $\psi_1^{-1}(v)=Q_p\cap Y_A$, which means that it is described on $Q_p$ as a Lagrangian degeneracy locus coming from restricting to $Q_p$ the degeneracy locus description of $Y_A$. To understand the latter restriction observe that $p$ defines a trivial subbundle $P \subset \mathcal T$ whose fibers over each [U] is the $1$-dimensional space $\langle p\rangle $ spanned by $p$. We hence get a bundle $\bar{\mathcal T}=\mathcal T/P$ Lagrangian in the $18$-dimensional space $\langle p\rangle ^{\perp}/\langle p\rangle $ with symplectic form induced from $\bigwedge^3 W$. Note that $A/\langle p\rangle $ is then a general Lagrangian subspace of $\langle p\rangle ^{\perp}/\langle p\rangle $. Indeed, each Lagrangian subspace of $\langle p\rangle ^{\perp}/\langle p\rangle $ extends to a Lagrangian subspace of $\bigwedge^3 W$ containing  $\langle p\rangle $.  
Now, since $\langle p\rangle \subset T_U$  the condition $\dim (A\cap T_U)\geq 2$ defining $Y_A$ as a second degeneracy locus, after restriction to $Q_p$ describes the first degeneracy locus associated to $\bar{\mathcal{T}}$. It follows that $Q_p\cap Y_A$ is the first Lagrangian degeneracy locus $D^1_{\bar A, \bar{\mathcal T}}$ associated to the bundle $\bar {\mathcal T}$ and a general Lagrangian subspace $\bar A:=A/\langle p\rangle $ of $\langle p\rangle ^{\perp}/\langle p\rangle $.

Note that generality implies in particular that this first degeneracy locus is not the whole $Q_p$ hence it is a hypersurface in $Q_p$ as expected. 
Since $c_1 (\bar{\mathcal{T}}^{\vee})=c_1 (\mathcal T^{\vee})=4H$, where $H$ is the Pl\"ucker hyperplane section,   we obtain that $Q_p\cap Y_A$ is a quartic section of $Q_p$. Consequently, $\psi_1^{-1}(v)=Q_p\cap Y_A$ is isomorphic to the intersection of a quadric and a quartic.

We now claim that $Q_p\cap Y_A= D^1_{\bar A, \bar{\mathcal T}}$ is a nodal surface. Indeed $\bar{\mathcal T}$ gives an embedding of $Q_p$ inside $LG(9, \langle p\rangle ^{\perp}/\langle p\rangle )$. We will denote the image of that embedding by $\bar Q_p$. Furthermore a general choice of $\bar{A}$ determines for each $i\geq 1$ a universal degeneracy locus $\mathbb D^i_{\bar{A}}$ on $LG(9, \langle p\rangle ^{\perp}/\langle p\rangle )$ 
associated to the universal subbundle on $LG(9, \langle p\rangle ^{\perp}/\langle p\rangle )$. Different choices of $\bar A$ give translates of $\mathbb D^i_{\bar{A}}$ by the transitive action of the symplectic group $\operatorname{Sp}(\langle p\rangle ^{\perp}/\langle p\rangle)$ on $LG(9, \langle p\rangle ^{\perp}/\langle p\rangle )$ and hence from the Kleiman-Bertini theorem we deduce that for $\bar A$ general the intersections of $\bar{Q}_p$ with all degeneracy loci $\mathbb D^i_{\bar A, \bar{\mathcal T}}$ are transversal. But these intersection are isomorphic to the corresponding degeneracy loci $D^i_{\bar A, \bar{\mathcal T}}$. It follows that $D^1_{\bar A, \bar{\mathcal T}}$ is a quartic section of $Q_p$ that is smooth outside a finite set of points $D^2_{\bar A, \bar{\mathcal T}}$ and $D^3_{\bar A, \bar{\mathcal T}}=\emptyset$. To compute the number of singular points we use the formula of Pragacz and Ratajski \cite[Theorem 2.1]{PR} applied to $\bar{\mathcal{T}}$ restricted to $Q_p$. This is done using Schubert calculus on $G(3,W)$ and gives 20. 

Finally, the double cover $\widetilde{Y}_A\to Y_A$ restricts to the preimage of $Q_p\cap Y_A$ as the double cover $\widetilde{Q_p\cap Y_A}\to Q_p\cap Y_A$ branched in the 20 singular points obtained from the general construction of double covers of Lagrangian degeneracy loci described in \cite[Section 4]{DK double covers}. We conclude by \cite[Corollary 4.8]{DK double covers}  that  $\widetilde{Q_p\cap Y_A}$ is smooth and the singularities of $Q_p\cap Y_A$ are nodes. 
\end{proof}
\begin{rem}
By Lemma \ref{fibers} the incidence correspondence $K_A$ provides a $4$-dimensional family $\{S_{x}\}_{x\in X_A}$ of surfaces on $Y_A$ parametrized by $X_A$.
 The preimage $\widetilde S_x$ by the double cover $\widetilde Y_A\to Y_A$ of a general surface $S_x$ in the family is a smooth surface whose canonical system defines the double cover of a 20 nodal complete intersection of a quadric and quartic branched at the nodes. Clearly the family $\{\widetilde S_x\}_{x\in X_A}$ dominates the double EPW cube $\widetilde Y_A$. Interesting special examples of such $20$ nodal surfaces were recently studied in a different context in \cite{B+}. 
 Note that such surfaces are rather special \cite{CPT} and are called canonical covers. 
 Recall for instance that the Schoen surface \cite{CMR} is the double cover of a complete intersection of a quadric and a quartic singular at $40$ nodes.
\end{rem}

We will consider the following diagram, where $\mu_X\,\colon\,\overline{X}_A\to X_A$, $\mu_Y\,\colon\,\overline{Y}_A\to X_A$ and $\mu_K\,\colon\,\widehat{K}_A\to K_A$ are resolution of singularities, and $\widehat{X}_A$ is obtained as a resolution of the fibre product of the leftmost square (similarly for $\widehat{Y}_A$ and the rightmost square), and maps starting from $\widehat{X}_A$ and $\widehat{Y}_A$ are the canonical ones of the fibre product. In the diagram, we replaced any resolution $\overline{K}_A$ of $K_A$ with the resolution of several fibre products (so that we obtain a smooth scheme which maps both to $\overline{X}_A$ and $\overline{Y}_A$), and $\widehat{\pi}_1$ and $\widehat{\psi}_1$ are obtained in this way.

\begin{equation}\label{diagram}
    \begin{tikzpicture}
  \matrix (m) [matrix of math nodes,row sep=3em,column sep=4em,minimum width=2em]
  {
      \widehat{X}_A& \overline{X}_A & \widehat{K}_A & \overline{Y}_A & \widehat{Y}_A \\
      \widetilde{X}_A & X_A & K_A & Y_A & \widetilde{Y}_A\\};

  \path[-stealth]
    (m-1-1) edge node [above] {$\widehat{p}_X$} (m-1-2)
    (m-1-1) edge node [right] {$\widetilde{\mu}_X$} (m-2-1)
    (m-2-1) edge node [above] {$p_X$} (m-2-2)
    (m-1-2) edge node [right] {$\mu_X$} (m-2-2)
    (m-1-3) edge node [above] {$\widehat{\psi}_1$ } (m-1-2)
            edge node [right] {$\mu_K$ }(m-2-3)
            edge node [above] {$\widehat{\pi}_1$ } (m-1-4)
    (m-2-3) edge node [above] {$\psi_1$} (m-2-2)
            edge node [above] {$\pi_1$ } (m-2-4)
    (m-1-4) edge node [right] {$\mu_Y$ } (m-2-4)
    (m-1-5) edge node [above] {$\widehat{p}_Y$ } (m-1-4)
            edge node [right] { $\widetilde{\mu}_Y$ } (m-2-5)
    (m-2-5) edge node [above] {$p_Y$ } (m-2-4);
\end{tikzpicture}
\end{equation}

By assumption \eqref{eq:Lagrangian_condition} on $A$, both $\widetilde{X}_A$ and $\widetilde{Y}_A$ are smooth,

\begin{prop} \label{prop:corr} Let $A_0$ be a general Lagrangian corresponding to a point in $LG(10, \bigwedge^3 W)\setminus (\Sigma\cup \Delta \cup \Gamma)$.
Then there exists a contractible open set $\mathcal {U}_0$ (in the standard topology), neighbourhood of $A_0$ contained in $LG(10, \bigwedge^3 W)\setminus (\Sigma\cup \Delta \cup \Gamma)$ over which there exists universal families $$\pi_{X,{U}}: \widetilde{\mathcal X}_{{U}} \to {U} $$
$$\pi_{Y,{U}}: \widetilde{\mathcal Y}_{{U}} \to {U}$$ 
of double EPW sextics and double EPW cubes respectively and a nowhere vanishing correspondence $$k_{{U}}: R^4{\pi_{X,{U}}}_* \mathbb Q\to R^4{\pi_{Y,{U}}}_* \mathbb Q$$ such that for all $A\in \mathcal {U}_0$
the map $k_A$ induced on the fibers over $A$ by $k_{{U}}$
is  a map of rational Hodge structures $H^4(\widetilde{X}_A,\mathbb{Q})\to H^4(\widetilde{Y}_A,\mathbb{Q})$\footnote{The map, being induced by pullbacks and pushforwards of maps between smooth schemes is actually a map of integral Hodge structures, but integrality will be in any case lost in Lemma \ref{lem:H4toH2}, so we keep rational coefficients}. 
\end{prop}

\begin{proof} We start by finding an open subset $\mathcal U_0$ such that we can produce Diagram (\ref{diagram}) relatively over $ \mathcal U_0$. Let us start with $\mathcal U_0$ a small contractible open neighborhood of $A_0$ in  $LG(10, \bigwedge^3 W)\setminus (\Sigma\cup \Delta \cup \Gamma)$
such that the families $\mathcal X_{ U}:=\{X_A\}_{A\in  U}$, $\mathcal Y_{ U}:=\{Y_A\}_{A\in  \mathcal U_0}$, $\mathcal K_{ U}:=\{K_A\}_{A\in \mathcal U_0}$ are all locally trivial. This can be done since $A_0$ is general.   We can then apply \cite[Lemma 4.9]{BL} to produce a simultaneous resolution of the families $\mathcal{X}_{{U}}, \mathcal{Y}_{{U}}$ and $\mathcal{K}_{{U}}$. As the covers $p_X$ and $p_Y$ are uniquely determined for $A\in \mathcal U_0$, we also obtain families of double covers $\widetilde{\mathcal{X}}_{{U}}, \widetilde{\mathcal{Y}}_{ U}$, which in turn gives us fiber products in families. In this way we get the leftmost and rightmost square in the diagram. Next, whenever for the diagram we need to perform a resolution of singularities we can shrink $\mathcal U_0$ further and use generality of $A_0$ to get local triviality of the resolved family.

We hence have Diagram \ref{diagram} defined relatively over $\mathcal U_0$. We then define  $$k_{{U}}:=\widetilde{\mu}_{Y*} \circ \widehat{p}^*_Y\circ \widehat{\pi}_{1*}\circ {\widehat{\psi}_1}^* \circ \widehat{p}_{X*} \circ \widetilde{\mu}_X^*,$$ which is a map of local systems.
The maps are well defined as we are only dealing with relatively smooth schemes. Therefore the correspondence $k_{{U}}$ exists and is fiberwise a map of rational Hodge structures. It is nowhere vanishing by Lemma \ref{nontrivial} below. Indeed, the lemma proves that a $(4,0)$-form in the fiber over $A_0$ has non-zero image and hence by possibly shrinking $\mathcal {U}_0$ we get that $k_{U}$ is non-zero for each fiber over $\mathcal {U}_0$. 
\end{proof}

\begin{lem}\label{nontrivial} Let $\sigma$ be a volume form on the smooth part of $X_A$ and $\bar \sigma$ its pushforward to $X_A$. Then $\sigma'=\mu_{Y*}\widehat{\pi}_{1*} \widehat{\psi}_1^* \mu_X^* \bar{\sigma} $ restricted to the smooth part of $Y_A$ is a non-trivial $4$-form.
\end{lem}
\begin{proof} 
As we are only dealing with $\widehat{\pi}_1$ and $\widehat{\psi}_1$, let us denote them by $\pi$ and $\psi$ in this Lemma. Moreover, our computation is done on general points of $X_A$ and $Y_A$, so that $\mu_X$ and $\mu_Y$ are the identity, and we omit them from the notation. 
Let us choose a general point  $x\in Y_A$.  We will evaluate the form $\sigma'$ in a special $4$-vector $r_1\wedge r_2\wedge r_3\wedge r_4$ on the tangent space of  $Y_A$ in $x$. We have $$\sigma'(r_1\wedge r_2\wedge r_3\wedge r_4)=\sum_{p\in \pi^{-1}(x)} \sigma({\psi_p}_*({\pi^*}_p(r_1))\wedge{\psi_p}_* ({\pi^*}_p(r_2))\wedge{\psi_p}_*({\pi^*}_p(r_3))\wedge{\psi_p}_* ({\pi^*}_p(r_4))),$$
where ${\psi_p}_*$ is the tangent map to $\psi$ in $p$
whereas $\pi_p^*$ is the inverse of the tangent map to $\pi$ in $p$. The latter is well defined since $\pi$ is a local isomorphism around $p$. In fact, by the generality assumption and Lemma \ref{fibers pi} we have $\pi^{-1}(x)=\{p_1, p_2, p_3\}$ for some $p_1, p_2 , p_3\in U\subset K_A$, where $U$ is an open subset of $\widehat{K}_A$ in which $\mu$ is an isomorphism. Consider the three surfaces $S_i$ for $i=1,2,3$ defined as $\pi(\psi^{-1}(\psi(p_i)))$. These surfaces locally around $x$ are exactly the surface described in Lemma 
\ref{fibers}. We claim that any two of the three surfaces are in general transversal in $x$ i.e. their projective tangent spaces meet only in $x$. Indeed, as we have seen in the proof of Lemma \ref{fibers}, these surfaces are obtained as intersections of $Y_A$ with quadric threefolds $Q_{p_i}$ for some $p_i$ on the projective tangent space to $G(3,W)$ in $x$. The latter quadric threefolds are general hyperplane sections of $G(2,4)$, that is there are three Grassmannians $G_{p_i}\subset G(3,W)$ which themselves meet pairwise in lines. We conclude by observing that by the generality assumption the hyperplane sections defining $Q_{p_i}$ on $G_{p_i}$ do not contain the lines $G_{p_i}\cap G_{p_j}$ for $j \in\{1,2,3\}$,  $j\neq i$.

Let us now  choose two of the surfaces $S_i$ and call them $S_1$ and $S_2$. We now choose $r_1$ and $r_2$ general vectors in the tangent space of $Y_A$ in $x$, which are tangent to  $S_1$ and $S_2$ respectively. Then  ${\pi^*}_{p_1}(r_1)$ is tangent to the fiber $\psi^{-1}(\psi(p_1))$, hence $ {{\psi}_{p_1}}_*({\pi^*}_{p_1}(r_1))=0$. Similarly,  $ {{\psi}_{p_2}}_*({\pi^*}_{p_2}(r_2))=0$. Finally, $$\sigma'(r_1\wedge r_2\wedge r_3\wedge r_4)= \sigma({\psi_{p_3}}_*({\pi^*}_{p_3}(r_1))\wedge{{\psi}_{p_3}}_* ({\pi^*}_{p_3}(r_2))\wedge{\psi_{p_3}}_*({\pi^*}_{p_3}(r_3))\wedge{{\psi}_{p_3}}_* ({\pi^*}_{p_3}(r_4))).$$ The latter is non-zero. Indeed,  by generality of choice of $r_1$ and $r_2$ and transversality of both $S_1$ and $S_2$ with $S_3$, the plane $\langle r_1, r_2\rangle$ meets the tangent to $S_3$ only in $0$ and hence $\langle {\pi^*}_{p_3}(r_1),{\pi^*}_{p_3}(r_2)\rangle $ is not contracted via ${\psi_{p_3}}_*$. Now, if we choose $r_3$ and $r_4$  two general vectors tangent to $Y_A$ in $x$, then since $\psi$ is locally a submersion around $p_3$ and $\pi$ a local isomorphism, the vectors ${\psi_{p_3}}_*(\pi^*_{p_3}(r_3))$ and ${\psi_{p_3}}_*(\pi^*_{p_3}(r_4))$ are general vectors in the tangent space to $X_A$ in $\psi(p_3)$. Hence, from the fact that $\sigma$ is a volume form on the smooth locus of $X_A $, we get  $\sigma'(r_1\wedge r_2\wedge r_3\wedge r_4)\neq 0$ and hence $\sigma' $ is non-trivial.
\end{proof}



\begin{section}{Elliptic curves of minimal degree}\label{sec elliptic}
In this section, we describe the moduli spaces $M_{1,0}(\widetilde X_A, \beta_X)$ and $M_{1,0}(\widetilde Y_A, \beta_Y)$  of elliptic curves of minimal degree on the double EPW sextic $\widetilde{X}_A$ and the double EPW cube $\widetilde{Y}_A$ (cf.~\cite{NOb}). In both cases, these elliptic curves are preimages of special lines contained in the EPW variety which have the property that they meet its singular locus in four points. We can then compare the moduli spaces by means of the correspondence $$K_A=\{([U],[\alpha])\in  Y_A\times  \mathbb{P}(\textstyle\bigwedge^3 W) \mid\ [\alpha] \in \PP(T_U)\cap \PP(A)\cap \Omega\}$$ defined already in Section \ref{sec:correspondence} which via its projections relates the  EPW sextic $X_A$ with the EPW cube $Y_A$. 
More precisely, the correspondence $K_A$ gives a way to associate to a general point $x\in X_A$ a nodal surface $S_x$ contained in $Y_A$. We will see that for a line $l\subset X_A$ corresponding to an element of $M_{1,0}(\widetilde X_A, \beta_X)$ the intersection of all surfaces $S_x$ for $x\in l$ is a line $L_l\subset Y_A$ giving an element of $M_{1,0}(\widetilde Y_A, \beta_X)$.


Let us first consider the following family of lines in $\Omega\setminus G(3,W)$. Recall that $\Omega\setminus G(3,W)$ admits two fibration structures $\pi_1$ and $\pi_2$ over $\mathbb{P}(W)$ and $\mathbb{P}(W^{\vee})$ respectively. Let $\mathcal Z\subset G(2,\bigwedge^3 W)$ be the set of lines in $\Omega\setminus G(3,W)$ that project to lines via both fibrations $\phi_1\,\colon\, \Omega\setminus G(3,W)\to \mathbb P(W)$ and $\phi_2\,\colon\, \Omega\setminus G(3,W)\to \mathbb P(W^{\vee})$. Finally, for a general Lagrangian subspace $A$ of $\bigwedge^3 W$, let $\mathcal Z_A\subset G(2,\bigwedge^3 W)$ be the set of lines from $\mathcal Z$ which are contained in $\mathbb{P}(A)$. Note that when $\mathbb P(A)\cap G(3,W)=\emptyset$ the set $Z_A$ is closed.

\begin{lemma}\label{description of l}
A line $l\in \mathcal Z$ is of the form
\begin{equation}\label{l def}
l=\{[(\lambda_1 v_1+\lambda_2 v_2)\wedge \alpha]\in \mathbb P(\textstyle \bigwedge^3 W) \mid\ \lambda_1, \lambda_2 \in \mathbb C\},
\end{equation} 
for some $v_1,v_2\in W$ with $v_1\wedge v_2\neq 0$ and $\alpha\in \bigwedge^2 W$ of maximal rank. Conversely, for general $\alpha\in \bigwedge^2 W$, $v_1,v_2\in W$,  the line $l$ defined by (\ref{l def}) is an element of $\mathcal{Z}$. In particular, $\mathcal Z$ is unirational and irreducible. 
\end{lemma}
\begin{proof}
Let $l$ be spanned by the classes of two $3$-vectors $\eta_1, \eta_2\in \bigwedge^3 W$. Since $l\subset \Omega$ we have $\eta_1=v_1\wedge \alpha_1$ and $\eta_1=v_2\wedge \alpha_2$ for some $v_1,v_2\in W$ and $\alpha_1,\alpha_2 \in \bigwedge^2 W$. Clearly $v_1$, $v_2$ are then linearly independent since otherwise $l$ would be contracted by the fibration $\pi_1$. We can now extend $(v_1,v_2)$ to a basis $(v_1,v_2,v_3,v_4,v_5,v_6)$ of $W$ and choose $\alpha_1$, $\alpha_2$ such that 
$$\alpha_1=v_2\wedge w_1+ \beta_1, \quad \alpha_2=v_1\wedge w_2+ \beta_2,$$ 
where $w_1,w_2\in W$ and $\beta_1,\beta_2\in \bigwedge^2 \langle v_3,v_4,v_5,v_6\rangle$. Note also that since $l\cap G(3,W)=\emptyset$, we have $\beta_1,\beta_2\neq 0$.
Consider now a general linear combination $\lambda_1 \alpha_1+\lambda_2 \alpha_2$. By assumption, for each $(\lambda_1:\lambda_2)\in \mathbb P^1$ there exists $(\delta_1:\delta_2)\in \mathbb P^1$ such that 

$$(\delta_1 v_1+\delta_2 v_2)\wedge (\lambda_1 \alpha_1+\lambda_2 \alpha_2)=0.$$
But this, after further multiplication respectively by $v_1$ and $v_2$ gives:
$$ \delta_2 v_1\wedge v_2\wedge (\lambda_1 \beta_1+\lambda _2 \beta_2) =0     $$
and 
$$ \delta_1 v_1\wedge v_2\wedge (\lambda_1 \beta_1+\lambda _2 \beta_2)=0,$$
which implies that $\beta_1$ and $\beta_2$ are linearly dependent. Hence, by possibly rescaling one of the $2$-vectors we can assume $\beta_1=\beta_2$ and denote it by $\beta$. Then by putting
$$\alpha:=v_1\wedge w_2+v_2\wedge w_1+ \beta,$$
we get $$\eta_1=v_1\wedge\alpha, \quad \eta_2=v_2\wedge\alpha.$$
To see that we can choose $\alpha$ of maximal rank, observe that 
since the projection $\phi_2(l)$ of $l$ in $\mathbb P(W^{\vee})=\mathbb P(\bigwedge^5 W)$ is spanned by $\alpha\wedge \alpha \wedge v_1$, $\alpha\wedge\alpha \wedge v_2$ and is a line by definition of $\mathcal Z$, we infer that $\alpha\wedge \alpha \wedge v_1$, $\alpha\wedge\alpha \wedge v_2$ are linearly independent. It follows that $\alpha$ cannot be of rank 2. Notice furthermore that if $\alpha$ is of rank 4 then $\alpha\wedge \alpha$ is a non-zero totally decomposable $4$-form. Now $\alpha\wedge \alpha \wedge v_1$ and $\alpha\wedge\alpha \wedge v_2$ are non-zero totally decomposable linearly independent $5$-form, which implies $\alpha\wedge \alpha\wedge v_1\wedge v_2 \neq 0$. It follows that for $ \alpha':=\alpha + v_1\wedge v_2 $ we have $\alpha'\wedge \alpha'\wedge \alpha'=\alpha\wedge \alpha \wedge v_1\wedge v_2\neq 0$ and hence $\alpha'$ is of rank 6 as needed and we still have  $ \eta_1=v_1\wedge \alpha', \eta_2=v_2\wedge \alpha'$ as before. 
\end{proof}
\begin{rem}\label{ZI}
 Let us define the locus $\mathcal Z_I\subset \mathcal Z$ of lines which are isotropic with respect to the conformal symplectic form on $\bigwedge^3 W$. These are given by the additional condition $v_1\wedge v_2\wedge \alpha \wedge \alpha =0$. This equation defines an irreducible hypersurface in the open set of triples 
 $(v_1,v_2, \alpha)\in \mathbb P(W)\times \mathbb P(W) \times\mathbb P( \bigwedge^2 W )$ such that $v_1\wedge v_2\neq 0$ and $\alpha^3\neq 0 $. Indeed for any form $\alpha$ of maximal rank the locus of  $(v_1, v_2, \alpha)$ satisfies the condition $v_1\wedge v_2\wedge \alpha \wedge \alpha =0$ when $[v_1\wedge v_2]$ is an element of a general hyperplane section of $G(2,W)$  corresponding to $\alpha^2$. 
 We conclude that $Z_I$ which is parametrized by this open irreducible  subset of $\mathbb P(W)\times \mathbb P(W) \times\mathbb P( \bigwedge^2 W )$ is also irreducible.

 Notice now, that the space of pairs $(A,l)$ satisfying the condition $l\in \mathcal Z_A$ is fibered over the irreducible locus $\mathcal Z_{I}$ of isotropic lines in $\Omega\setminus G(3,W)$ with fibers isomorphic to the Lagrangian Grassmannians $LG(8,16)$. Indeed $l\subset  \PP(A)$ implies $l$ is isotropic. Furthermore for a fixed isotropic $l$ the locus of Lagrangian subspaces $\mathbb{P}(A)$ containing $l$ is isomorphic to a Lagrangian Grassmannian $LG(8,\langle l\rangle^{\perp}/\langle l\rangle)$ with respect to the induced symplectic form, where $\langle l\rangle$ is the  $2$-dimensional isotropic subspace corresponding to l.  We conclude that the variety of pairs $\{(A,l)\mid l\in \mathcal Z_A \}$ is irreducible. 
It follows that we can speak of general pairs $(A,l)$ such that  $l\in \mathcal Z_A$. Note that we do not claim that $\mathcal Z_A$ is irreducible.
\end{rem}

The following properties of  $\mathcal Z_A$ and lines represented by elements of $\mathcal Z_A$ will be used.
\begin{lemma}\label{properties of Z_A} Let $(A,l)$ be a general pair such that $A$ is Lagrangian subspace of $\bigwedge^3 W$ and $l$ is a line in $\mathbb{P}(\bigwedge^3 W)$ such that $l\in \mathcal Z_A$. Then,
\begin{enumerate}
    \item $\phi_1(l)\subset X_A$ and $\phi_1(l)\cap \operatorname{Sing}(X_A) $ is a reduced $0$-dimensional scheme of length~4. 
    \item $\phi_2(l)\subset X_{A^{\vee}}$ and $\phi_2(l)\cap \operatorname{Sing}(X_{A^{\vee}})$ is a reduced $0$-dimensional scheme of length~4. 
    \item The locus $\{[U]\in G(3,W) \mid\ l\subset \mathbb P(T_U)\}$ is a line $L_l$.
    \item We have $L_l\subset Y_A$ and $L_l\cap \operatorname{Sing}(Y_A)$ is a reduced $0$-dimensional scheme of length~4.
  
\end{enumerate}
Furthermore the following converses hold:
\begin{itemize}
    \item For any line $k\subset X_A$ such that $k \cap D^2_{A,\mathcal F}$ consists of four points, there exists $l\in \mathcal Z_A$ such that $k=\phi_1(l)$.
    \item For any line $k\subset Y_A$ such that $k \cap D^2_{A,\mathcal F'}$ consists of four points, there exists $l\in \mathcal Z_A$ such that $k=\phi_2(l)$.
    \item  For any line $K\subset G(3,W)$ such that  $K\subset Y_A$ and $K\cap D^3_{A,\mathcal T}$ consists of four points, there exists $l\in \mathcal Z_A$ such that $K=L_l$.
\end{itemize}
 
\end{lemma}
\begin{proof}
Let $G_A:=\Omega\cap \mathbb{P}(A)$. Note that (see \cite[Appendix]{GK}) $\phi_1|_{G_A}\,\colon\, G_A \to X_A$ and $\phi_2|_{G_A}\,\colon\, G_A\to X_{A^{\vee}}$ are resolutions of singularities  such that $$(\phi_2|_{G_A}) \circ (\phi_1|_{G_A})^{-1}\,\colon\, X_A\dashrightarrow X_{A^{\vee}} \text{ and }(\phi_1|_{G_A}) \circ (\phi_2|_{G_A})^{-1}\,\colon\,  X_{A^{\vee}}\dashrightarrow X_A$$ are maps defined by partial derivatives of the sextic polynomials defining the corresponding EPW sextics. Now, clearly, since $l\subset G_A$, we have $\phi_1(l)\subset X_A$. Moreover, by definition of $\mathcal Z_A$, both $\phi_1(l)$ and $\phi_2(l)$ are lines. Moreover, by Remark 4.2 we can assume that both these lines are not contained in the sinugular loci of $X_A$ and $X_{A^{\vee}}$. Indeed, we can first choose $l$ in $\mathcal Z_I$ and then find $A$ for which these assumptions are satisfied. It follows that the linear system of restrictions of quintic partial derivatives of the equation of $X_A$ to $\pi_1(l)$, which defines the  map
$$(\phi_2|_{G_A}) \circ (\phi_1|_{G_A})^{-1}|_{\pi_1(l)}\,\colon\, \phi_1(l)\to \phi_2(l),$$ must have a base locus of degree 4. Consequently, $\phi_1(l)\cap D^2_{A,\mathcal F}$ is a $0$-dimensional scheme of degree 4. 

It is now enough to check on an example with Macaulay 2 (see Appendix) that for a general pair $(A,l)$ with $l\in \mathcal Z_A$ (cf. Remark \ref{ZI}) this scheme consists of four points. This proves (1), the proof of (2) is analogous.

For the converse of (1) we use again \cite[Appendix]{GK}. The proper transform of $k$ on $G_A$ is a curve $\widetilde k$ with the property $E\cdot \widetilde k=4$ and $(5H-E)\cdot \widetilde k=1$, where $H$ is the pullback to $G_A$ of the hyperplane class on $\mathbb{P}(W)$ and $E$ is the exceptional divisor of the resolution $ \phi_1|_{G_A}$. Recall now from \cite[Appendix, Corollary A.7]{GK} that in our notation the pullback to $G_A$ of the hyperplane class on $\mathbb{P}(W^{\vee})$ is $5H-E$ and the restriction of the hyperplane class from $\PP(\bigwedge^3 W)$ to $G_A$ is then $\frac{1}{2} (6H-E)$.   It follows that $\widetilde k$ is a line in $G_A$ that maps to lines via each of the maps $\phi_1$ and $\phi_2$ and  consequently $\widetilde k\in \mathcal Z_A$. Analogously we obtain the proof of the converse of (2).

Let us now pass to the proof of (3). By Lemma \ref{description of l} we have $$l=\{[(\lambda_1 v_1+\lambda_2 v_2)\wedge \alpha]\in \mathbb P(\textstyle{\bigwedge}^3 W) \mid\ \lambda_1, \lambda_2 \in \mathbb C\}, $$
for some linearly independent $v_1,v_2\in W$ and $\alpha\in \bigwedge^2 W$ of maximal rank.  Let $V$ be the subspace of $W$ spanned by $v_1,v_2$. Consider also the $5$-vectors $v_1\wedge \alpha\wedge \alpha\in \bigwedge^5 W$ and $v_2\wedge \alpha\wedge \alpha\in \bigwedge^5 W $ which are linearly independent since $l\in \mathcal Z_A$. As $\dim W=6$ these are totally decomposable $5$-vectors and hence correspond to hyperplanes in $\mathbb{P}(W)$, furthermore these hyperplanes intersect in a $4$-dimensional space that we will denote by  $P\subset W$. Then $V\subset P$ since $l\in \mathcal Z_A$ implies that $l$ is an isotropic line and hence $v_1\wedge v_2 \wedge \alpha \wedge \alpha=0$ which means that both $v_1$ and $v_2$ are in both 5-spaces corresponding to the $5$-vectors $v_1\wedge \alpha\wedge \alpha$ and $v_2\wedge \alpha\wedge \alpha$. It follows that the space $$\{[U]\in G(3,W)\mid\ V\subset U\subset P\}$$ is a line that we claim to be $L_l$. Indeed, by  (\ref{pinTU}), we know that $v_1\wedge \alpha\in T_U$ and $v_2\wedge \alpha\in T_U$ if and only if $v_1,v_2\in U$ and $\alpha\wedge \bigwedge^3U=0$, which is equivalent to $V\subset U\subset P$.

For (4) we observe that $L_l\subset Y_A$ follows directly from the definition since for any $[U]\in L_l$ we have $l\subset \mathbb P(T_U)$ and $l\subset \mathbb P(A)$ which gives $\dim (T_U\cap A)\geq 2$ i.e. $[U]\in Y_A=D^2_{A,\mathcal T}$. To describe $L_l\cap \operatorname{Sing} Y_A= L_l\cap D^3_{A,\mathcal T}$ it is enough to see that this locus is described on $L_l$ as a first Lagrangian degeneracy locus. Indeed, take the $2$-dimensional vector space $V_l\subset \bigwedge^3 W$ corresponding to $l$ and endow $V_l^{\perp}/V_l$ with the induced conformal symplectic structure. Then $\bar{\mathcal T}_l=\{T_U/V_l\}_{[U]\in L_l}$ is a family of Lagrangian spaces parametrized by $L_l$. Furthermore, by generality assumption we have $A/V_l$ is a general Lagrangian space. To see it we observe that Remark \ref{ZI} allows us to first choose a general $l\in\mathcal Z_I$ and then a general Lagrangian space $A$ such that $l\subset \mathbb P(A)$. The latter choice of $A$ is equivalent to the choice of a Lagrangian subspace in $V_l^{\perp}/V_l$.  

Having generality of $A/V_l$ we can use the Kleiman-Bertini theorem to obtain that $D^1_{A/V_l,\bar{\mathcal T}_l  }$ is a smooth set of four points. Indeed $\bar{\mathcal T}_l $ gives an embedding of $L_l$ into the Lagrangian Grassmannian $LG( 8, V_l^{\perp}/V_l) $ on which we have a transitive action of the symplectic group $Sp(V_l^{\perp}/V_l)$. Now $A/V_l$ defines a universal Lagrangian degeneracy locus $\mathbb D^1_{A/V_l}=D^1_{A/V_l, \mathcal U_{LG( 8, V_l^{\perp}/V_l)}}$ associated to the universal bundle $\mathcal U_{LG( 8, V_l^{\perp}/V_l)}$ seen as a subbundle of the trivial bundle with fibers $V_l^{\perp}/V_l$. Then $\mathbb D^1_{A/V_l}$ is a generically smooth hypersurface and different choices of $A/V_l$ lead to its translates. By the Kleiman-Bertini theorem we can choose $A/V_l$ such that the intersection of the image of $L_l$ in $LG( 8, V_l^{\perp}/V_l) $  via the embedding given by $\bar{\mathcal T}_l $ meets transversely $\mathbb D^1_{A/V_l}$. The intersection will then be a smooth finite scheme. 



To compute the number of points we observe that the bundle with fibers $T_U/V_l$ is a quotient of the restriction $\mathcal T|_{L_l} $ by a trivial subbundles and hence its Chern classes are the same as restrictions of Chern classes of $\mathcal T$ to $L_l$. In consequence the degree of the first Lagrangian degeneracy locus is four which concludes the proof that $L_l\cap \operatorname{Sing} (Y_A)$ is a set of four distinct points.

For the converse, let us be more precise. Consider a line $K$ in $G(3,W)$. Let $V_K$ and $P_K$ be the $2$- and $4$-dimensional subspaces of $W$ such that $$K=\{[U]\in G(3,W) \mid\ V_K\subset U\subset P_K\}.$$ 
Consider a corresponding decomposition $W=V_1+V_2+V_3$ into a sum of $2$-dimensional subspaces such that $V_1=V_K$ and  $V_1+V_2=P_K$.
Let 
$$\textstyle R_K:=\bigwedge^2 V_K\wedge W + \bigwedge^3 P_K= \bigwedge^2 V_1\wedge V_2+ \bigwedge^2 V_1\wedge V_3+ V_1\wedge \bigwedge^2 V_2 $$
and observe that $R_K$ is a $6$-dimensional vector subspace of $\bigwedge^3 W$ such that $$R_K=\bigcap_{[U]\in K} T_U.$$ 
Note that $R_K$ is isotropic with respect to the conformal symplectic structure on $\bigwedge^3 W$. Consider $R_K^{\perp}/R_K$ with the induced conformal symplectic structure and the $1$-dimensional family of Lagrangians $\bar{\mathcal T}=\{T_U/R_K\}_{[U]\in K}$. We will denote by $\bar{\mathcal T}_{[U]}:=T_U/R_K$ the fiber of $\bar{\mathcal T}$ over $[U]$. 
Observe that $R_K^{\perp}=\bigwedge^2 V_1\wedge V_2+ \bigwedge^2 V_1\wedge V_3+ V_1\wedge \bigwedge^2 V_2 + V_1\wedge V_2\wedge V_3$, hence $R_K^{\perp}/R_K$ is naturally identified to $ V_1\wedge V_2\wedge V_3$ with conformal symplectic structure defined by wedge product. Under this identification $\bar{\mathcal T}_{[U]}$ for $V_1\subset U\subset V_1+V_2$ is the space $V_1\wedge (U\cap V_2)\wedge V_3$.  
It follows that $\bar {\mathcal T}$ gives a family of disjoint $3$-dimensional projective spaces sweeping a $\mathbb{P}^1\times \mathbb{P}^3\subset \mathbb{P}^7=\mathbb P(R_K^{\perp}/R_K)$.

For a Lagrangian space $A\subset \bigwedge^3 W$
let us consider $\bar{A}=(A\cap R_K^{\perp})/(A\cap R_K)$. Set $c=\dim (A\cap R_K)$ and observe that 
$$K\cap D^3_{A,\mathcal T}=D^{3-c}_{\bar A, \bar{\mathcal T}}.$$
We want to prove that whenever $K\cap D^3_{A,\mathcal T}$ consists of four points then $\mathbb{P}(A\cap R_K)$ is a line $l$ (i.e. $c=2$) that necessarily is represented by an element of $\mathcal Z_A$ (since $R_K\subset \Omega$) and satisfies $L_l=K$. 
For that it is enough to prove that $D^{3-c}_{\bar A, \bar{\mathcal T}}$ cannot consist of four points if $c\neq 2$. 
We have two possibilities to exclude:
\begin{itemize}
    \item If $D^{3}_{\bar A, \bar{\mathcal T}}$ consists of four points then these four points correspond to four fibers $\mathcal T_{[U_i]}$ for $i=1\dots 4$ such that $\dim (\mathbb P(\bar A)\cap \mathbb P(\mathcal T_{[U_i]}))\geq 2$. But for any $i\neq j$ we have $\mathbb P(\mathcal T_{[U_i]}) \cap \mathbb P(\mathcal T_{[U_j]})=\emptyset$. In particular, $\mathbb P(\bar A)$, which is a projective space of dimension 3,  must contain two disjoint planes. This is a contradiction.

\item If $D^{2}_{\bar A, \bar{\mathcal T}}$ consists of four points then these four points correspond to four fibers $\mathcal T_{[U_i]}$ for $i=1\dots 4$ such that $\dim (\mathbb P(\bar A)\cap \mathbb P(\mathcal T_{[U_i]}))\geq 1$. Let us then take a line $k$ contained in the intersection  $\mathbb{P}({\bar{\mathcal T}}_{[U_1]})$. We observe that its orthogonal complement $k^{\perp}$ is a $\mathbb{P}^5$ containing $\mathbb P({\bar{\mathcal T}}_{[U_1]})$ and meeting $\bigcup_{[U]\in K}\mathcal T_{[U]}=\mathbb{P}^1\times \mathbb{P}^3\subset \mathbb{P} ^7$ in the union of $\mathbb P({\bar{\mathcal T}}_{[U_1]})$ and a residual quadric $\mathbb{P}^1\times\mathbb{P}^1$.  Now $\mathbb P(\bar A)$ is a projective 3-space contained in $k^{\perp}$ that contains two (in fact four) disjoint lines lying on the quadric and hence contains the whole quadric. But this means that 
$\mathbb P(\bar A)$ meets $\mathbb P({\bar{\mathcal T}}_{[U]})$ in at least a line for each $[U]\in K$. In consequence $D^{2}_{\bar A, \bar{\mathcal T}}=K$ which is a contradiction with our assumption that it consists of exactly four points.
\end{itemize}
 We conclude that under our assumptions on $K$ we must have $c=2$. 
\end{proof}

\begin{lemma}\label{curve ZA} When $A$ is general, the variety   $\mathcal Z_A\subset G(2,\bigwedge^3 W)$ is of pure dimension 1.
\end{lemma}

\begin{proof} We first compute the dimension of $\mathcal Z$ by computing the dimension of the locus of lines in $\mathcal Z$ which are contained in a fixed space $\PP( T_{U})$. For each $[U]\in G(3,W)$ we will denote this locus by $Z_U$. The elements of $Z_U$ are lines contained in $(\Omega\setminus G(2,W))\cap \mathbb P( T_U)$ which is the smooth locus of a determinantal cubic cone $R_U$. Now observe that, inside a determinantal cubic (cone) in $\PP^9$ there are three types of lines that do not intersect the singular locus $C_U=C(\PP^2\times \PP^2)$: those contained in the $\PP^6$ span of two $\PP^3$ in one of the two rulings of $C_U$ (these $\PP^3$ meet only in the vertex of the cone)  or those contained in the $\PP^5$ span of two $\PP^3$ from different rulings of $C_U$ (these $\PP^3$ meet in a line passing through the vertex of the cone). 
The latter are exactly the lines in  $\mathcal Z_U$. 
Now, for general $[U]\in G(3,W)$ the locus $Z_{U}$ is $12$-dimensional, since the choice of a general line in $\PP^5$ is $8$-dimensional and the choice of a pair of $\PP^2$ fibers from different rulings in $\mathbb{P}^2\times \mathbb{P}^2$ is $4$-dimensional.

From Lemma \ref{properties of Z_A}(3) we deduce
$$\mathcal Z=\bigcup_{[U]\in  G(3,W)}Z_U \subset \operatorname G(2,\textstyle \bigwedge^3 W).$$
Let us now consider the correspondence 
$$\bar{\mathcal Z}=\{(l,[U])\in \mathcal Z\times G(3,W) \mid\ \ l\subset R_U \}.$$
Then $\bar{\mathcal Z}$ admits two dominant projections: onto $\mathcal Z$ and $G(3,W)$. The fiber over $[U]$ of the projection to $G(3,W)$ is $Z_U$ and is $12$-dimensional as explained above. Whereas each fiber of the first projection is a line as seen in Lemma \ref{properties of Z_A}(3) (observe that the generality assumption on l was not used in Lemma \ref{properties of Z_A}(3)). It follows that $\bar{\mathcal Z}$ is irreducible of dimension 12+9=21 and in consequence $\mathcal Z$ is irreducible and has dimension $21-1=20$.

Next, we consider the correspondence $$\textstyle \mathcal D=\{ (A,l,[U])\in (\mathrm{LG}(10,\bigwedge^3W) \setminus \Sigma)\times \mathcal Z\times G(3,W) \mid \ \ l\subset A,  \ l\subset R_U \},$$
with two natural projections: \begin{itemize}
    \item $p_1\,\colon\, \mathcal D \to \mathrm{LG}(10,\bigwedge^3W)\setminus \Sigma$,  whose fibers are $\mathcal Z_A$, for  $A\in \mathrm{LG}(10,\bigwedge^3W)\setminus \Sigma$;
    \item  $p_2\,\colon\, \mathcal D \to \bar{\mathcal Z}$, whose fibers are all isomorphic to open subsets of $LG(8,16)$.
\end{itemize} 
It follows that $\mathcal D$ is irreducible of dimension $\dim LG(8,16)+\dim \bar{\mathcal Z}=36+20=56$. Now since $\dim LG(10,\bigwedge^3 W)=55$ we get $\dim \mathcal Z_A\geq 1$

It remains to show that for a general $A$ the family $\mathcal Z_A$ is at most $1$-dimensional.  
For this, we consider the dimension of the tangent space of $\mathcal Z_A$ in points $l\in \mathcal Z_A$ such that the pair $(A,l)$ is general with this property. In other words, for 
  $l\in \mathcal Z_A$ general, we compute the dimension of the space of sections of the normal bundle $N_{l/\Omega\cap \PP(A)}$.

For that, observe that when $A$ is a general Lagrangian space, then $X_A$ is an EPW sextic singular in a smooth surface $S_A$ of degree 40. In that case $\Omega\cap \mathbb P(A)$ is the blow up of $X_A$ along $S_A$. Let us denote by $\pi: \widetilde{\mathbb{P}(W)} \to \mathbb{P}(W)$ the blow up of $\mathbb{P}(W)$ in $S_A$. To simplify notation, in the following, we will write $\mathbb P^5$ for $\mathbb P(W)$ and $\widetilde{\mathbb P^5}$ for $\widetilde{\mathbb{P}(W)}$. Then $\Omega\cap \mathbb P(A)$ is a smooth divisor $D$ in $\widetilde{\mathbb P^5}$ which belongs to the system $|6H-2E|$, where $H$ is the pullback of the hyperplane class in $\mathbb{P}^5$ via $\pi$, whereas $E$ is the exceptional locus of $\pi$. Observe that $\pi(l)$ is a line in $\mathbb{P}^5$ meeting $S_A$ in four points. It follows that $H\cdot l=1$ and $E\cdot l=4$. In consequence, $\mathcal O_l(6H-2E)\simeq \mathcal O_{\mathbb P^1}(-2)$. Let us consider the following exact sequence for the blow up $\pi$:

$$0\to T_{\widetilde{\mathbb P^5}} \xrightarrow{\kappa} \pi^* T_{\mathbb P^5}\to T_{\pi|_E}(E)\to 0,$$
where $\kappa$ is the tangent map to $\pi$ and $T_{\pi|_E}$ is the pushforward via the natural inclusion of the  relative tangent bundle of the restricted map  $\pi|_E$.  
Restricting the sequence to $l$ and taking into account that $T_l\subset T_{\widetilde{\mathbb P^5}}$ maps via $\kappa$ isomorphically to $T_l\subset \pi^* T_{{\mathbb P^5}}$,  gives
\begin{equation}\label{firstnorm}
    0\to N_{l/\widetilde{\mathbb{P}^5}}\xrightarrow{\bar{\kappa}} \pi|_l^*(N_{\pi(l)/\mathbb P^5}) \to T_{\pi|_{l\cap E}}(E)\to 0.
\end{equation}

We furthermore have 
\begin{equation}\label{norm}
0\to N_{l/D}\to N_{l/\widetilde{\mathbb{P}^5}}\to N_{D/\widetilde{\mathbb{P}^5}}|_l\to 0.
\end{equation}

Observe that $N_{D/\widetilde{\mathbb{P}^5}}|_l \simeq \mathcal O_{\mathbb P^1}(-2)$
and $\pi|_l^*(N_{\pi(l)/\mathbb P^5})= 4\mathcal O_{\mathbb P^1}(1)$. Furthermore $T_{\pi|_{l\cap E}}(E)=\bigoplus_{i=1}^4 \mathbb C_{p_i}^{\oplus 2}$, where $p_1,\dots, p_4$ are the four points of intersection of $l$ with $E$ and $\mathbb C_{p_i}$ is the skyscraper sheaf supported on $p_i$. 
Now, from (\ref{norm}) we have $c_1(N_{l/D})=-2$ and in consequence $h^0(N_{l/D})\geq 1$. Furthermore, the long exact sequence of cohomology associated to (\ref{norm}):
$$0\to H^0(N_{l/D})\to H^0(N_{l/\widetilde{\mathbb P^5}}) \to H^0(N_{D/\widetilde{\mathbb P^5}}|_l)\to H^1(N_{l/D})\to 
H^1(N_{l/\widetilde{\mathbb{P}^5}})\to H^1(N_{D/\widetilde{\mathbb P^5}}|_l)\to 0,$$
gives
 $h^0(N_{l/D})=h^0(N_{l/\widetilde{\mathbb{P}^5}})$ and $h^1(N_{l/D})=h^1(N_{l/\widetilde{\mathbb{P}^5}})-1$.
Additionally, the long exact sequence of cohomology associated to (\ref{firstnorm}):
$$0\to H^0(N_{l/\widetilde{\mathbb{P}^5}})\to H^0(N_{\pi(l)/\mathbb P^5}) \to H^0(T_{\pi|_{l\cap E}})\to H^1(N_{l/\widetilde{\mathbb{P}^5}})\to 0,$$
gives $h^0 (N_{l/\widetilde{\mathbb{P}^5}})- h^1 (N_{l/\widetilde{\mathbb{P}^5}})=0$.

We claim that $N_{l/D}=\mathcal O_{l}(-1)\oplus \mathcal O_l(-1)\oplus \mathcal O_l$. By the above it is enough to prove that $h^0(N_{l/\widetilde{\mathbb{P}^5}})=1$.
To prove that, first observe that elements in $H^0 (N_{l/\widetilde{\mathbb{P}^5}})$ correspond to elements of $H^0(N_{\pi(l)/\mathbb P^5})$ which are in the image of $\bar{\kappa}$. On the other hand, by the tangent bundles sequence associated to $\pi(l)\subset \mathbb{P}^5$, elements of $H^0(N_{\pi(l)/\mathbb P^5})$ represent classes of sections of  $T_{\mathbb P^5}|_{\pi(l)}$ in the space of sections of $T_{\pi(l)}$. It follows that sections of $N_{\pi(l)/\mathbb P^5}$ correspond to maps $$(x_1:x_2)\mapsto F(x_1:x_2)=F_1(x_1,x_2)\frac{\partial}{\partial x_3}+\dots+ F_4(x_1,x_2)\frac{\partial}{\partial x_6},$$ where $(x_1,x_2)$ are coordinates of $\pi(l)$ and $x_3,\dots,x_6$ are coordinates of the projective space $\mathbb{P}^3=\mathbb P(W/\langle\pi(l)\rangle)$. Such a section is in the image of $\bar\kappa$ if and only if for $i\in\{1,\dots ,4\}$ we have $(F_1(p_i):\ldots :F_4(p_i))\in \bar T_{S_A,p_i}$ where $\bar T_{S_A,p_i}$ is the image of the tangent space to $S_A$ in $p_i$ via the projection from $\pi(l)$. Observe now that the image $F(\pi(l))$ is then a line that intersects all four lines $\bar T_{S_A,p_i}$. The space of sections can be $2$-dimensional only if the four lines $\bar T_{S_A,p_i}$ lie in a quadric surface. This is not the case in a chosen example as we check by the Macaulay 2 script in the Appendix. We conclude by Remark \ref{ZI} that $h^0(N_{l/D})=1$, $h^1(N_{l/D})=0$ and $N_{l/D}=2\mathcal O_l(-1)\oplus \mathcal O_l$ for a general pair $(A,l)$ with $l\in \mathcal Z_A$. 
Finally $\mathcal Z_A$ is $1$-dimensional
and smooth at $l$ for a general pair $(l,A)$ such that $l\in \mathcal Z_A$. In particular, for $A$ general $\mathcal Z_A$ is smooth in a general point of every of its components.  

\end{proof}

The following proves Conjecture 1.1 in \cite{NOb} in the case of fourfolds of $K3^{[2]}$ type and BBF degree $2$,  and sixfolds of $K3^{[3]}$ type and BBF degree $4$ and divisibility $2$. Recall that $\mathcal \beta_X$, $\mathcal \beta_{X^{\vee}}$ and $\mathcal \beta_Y $ denote classes of curves of minimal degree on $\widetilde X_A$, $\widetilde X_{A^{\vee}}$ and $\widetilde Y_A$ respectively for $A$ very general.
\begin{thm}\label{thm:elliptic}
 For a very general $A$, the moduli spaces $M_{1,0}(\widetilde X_A,\beta_X)$, $M_{1,0}(\widetilde X_{A^{\vee}},\beta_{X^{\vee}})$, $M_{1,0}(\widetilde Y_A,\beta_Y)$ are  $1$-dimensional and their reduced structures are isomorphic to each other.
\end{thm}
\begin{proof} 
Observe that for $l\in \mathcal Z_A$ the preimages of lines $\pi_1(l)\subset X_A$, $\pi_2(l)\subset X_{A^{\vee}}$ and $L_l\subset Y_A$, considered in Lemma \ref{properties of Z_A}, via the corresponding double covers $\widetilde X_A\to X_A$, $\widetilde X_{A^{\vee}}\to X_{A^{\vee}}$ and $\widetilde Y_A\to Y_A$ are elliptic curves which intersect the polarisation of their corresponding hyper-K\"ahler manifolds in degree 2. 
Using the BBF form, $H_2(\widetilde{X}_A,\ZZ)$ coincides with the dual lattice of $H^2(\widetilde{X}_A,\ZZ)$, and the same holds for $\widetilde{X}_{A^{\vee}}$ and $\widetilde{Y}_A$. Therefore, we can compute the intersection of a minimal integral curve in a very general double EPW cube (and sextic) with the polarization $H$ by taking a generator of the algebraic part of the dual lattice, which is given by $H$ itself for EPW sextics, and by $H/2$ for EPW cubes (as the polarization has divisibility 2). In both cases, we obtain $2$ as the intersection number. 
We conclude that homology classes of constructed elliptic curves are curve classes of minimal degree. 

Note  that in each case, a smooth elliptic curve having class of minimal degree must map via the double cover to a line 4-secant to the branch locus. Such lines, in each case, by Lemma \ref{properties of Z_A}, correspond one-to-one to elements of $\mathcal Z_A$. In consequence, the reduced structures of all considered moduli spaces are isomorphic to $\mathcal Z_A$.  Lemma \ref{curve ZA} implies that they are of pure dimension 1.

\end{proof}

\begin{rem}\label{rem:indet}
Observe that using $K_A$ and its projections one can construct an embedding of the singular locus of the EPW sextic $X_A$ 
(or equivalently $X_A^{\vee}$) into  the EPW cube $Y_A$. Indeed, a point in $p\in \operatorname{Sing}(X_A)$ gives rise to a line $l_p:=\phi^{-1}_1(p)\cap \mathbb{P}(A)\subset \Omega$ 
that we can show to be contained in the cubic $R_U=\PP(T_U)\cap \Omega$ for some unique $[U]\in G(3,6)$ and the map $p\mapsto [U]$ will be our embedding.
\end{rem}
\end{section}
\section{Hodge structures} \label{sec:hodge}
In this section, we prove that the correspondence $k_A$ that we have constructed in Proposition \ref{prop:corr} can be used to give an isometry of Hodge structures between the primitive parts of cohomologies $H^2_{{\operatorname{prim}}}(\widetilde{X}_A,\mathbb{Z})$ and $H^2_{{\operatorname{prim}}}(\widetilde{Y}_A,\mathbb{Z})$ as stated in the introduction.

Note that  the invariant part $H^4(\widetilde{X}_A,\mathbb Q)^{\iota_X}$  embeds into $H^4(X_A ,\mathbb Q),$ as $X_A$ is an orbifold obtained as the quotient of $\widetilde{X}_A$  by the covering involution $\iota_X$. This implies that on the smooth locus of $X_A$ there is a $4$-form that is the pushforward of the $4$-form $\sigma_{X_A}\wedge \sigma_{X_A}$ on $\widetilde{X}_A$, where $\sigma_{X_A}$ is a symplectic form on $\widetilde{X}_A$, which is anti-invariant by $\iota_X$. Recall that, from Proposition \ref{prop:corr}, we know that $k_A$ is a non-zero correspondence which gives us a nontrivial $4$-form on $\widetilde{Y}_A$ by taking the image of $\sigma_{X_A}\wedge \sigma_{X_A}$ via $k_A$ as proven in Lemma \ref{nontrivial}.
Furthermore, by \cite{V} and \cite[Theorem 6.2.4]{DeMi} since $\widetilde{X}_A$ is a deformation of a Hilbert square of a K3 surface we have 
$$H^4(\widetilde{X}_A,\mathbb Q)=\operatorname{Sym}^2 (H^2 (\widetilde{X}_A, \mathbb Q)).$$ Moreover, $$H^2_{{\operatorname{prim}}} (\widetilde{X}_A, \mathbb Q)\subset H^2 (\widetilde{X}_A, \mathbb Q) $$ is the anti-invariant part of the involution $\iota_X$. Indeed the invariant part of the involution $\iota_X$ is isomoprhic to $H^2(X_A,\mathbb Q)$ by the pullback map. However,  by the Lefschetz hyperplane theorem, $H^2(X_A,\mathbb Q)$ is generated by the class $H$, the natural polarization arising as the pullback of $\mathcal{O}(1)$ on $X_A$. The pullback of the latter to $H^2 (\widetilde{X}_A, \mathbb Q)$ is the class $L$ of the polarisation on $\widetilde{X}_A$ and hence its orthogonal is the primitive part of the cohomology by definition.  It follows that  $$H^4_{{\operatorname{prim}}}(\widetilde{X}_A)^{\iota_X}=\Sym^2 (H^2 _{{\operatorname{prim}}}(\widetilde{X}_A, \mathbb Q))\oplus \mathbb{Q}H^2,$$ 
In a similar way, we fix a symplectic form $\sigma_{Y_A}$ and we have an inclusion $\Sym^2 (H^2_{{\operatorname{prim}}}(\widetilde{Y}_A, \mathbb Q))\subset H^4_{{\operatorname{prim}}}(\widetilde{Y}_A)^{\iota_Y}$ where $\iota_Y$ is the covering involution of $\widetilde{Y}_A\to Y_A$.
Notice that the Hodge structures $\Sym^2 (H^2 _{{\operatorname{prim}}}(\widetilde{X}_A, \mathbb Q))$ and $\Sym^2 (H^2_{{\operatorname{prim}}} (\widetilde{Y}_A, \mathbb Q))$ are not irreducible. Indeed, by \cite[Proposition 3.2]{O1}  they contain Hodge substructures $W(X_A)$ and $W(Y_A)$ respectively, which are the orthogonal complements of $\mathbb{Q}\sigma_{X_A}\overline{\sigma_{X_A}}$ and  $\mathbb{Q}\sigma_{Y_A}\overline{\sigma_{Y_A}}$. Notice that for $A$ very general, $W(X_A)$ and $W(Y_A)$ are the smallest irreducible Hodge structures containing $\sigma_{X_A}\wedge \sigma_{X_A}$ or $\sigma_{Y_A}\wedge \sigma_{Y_A}$ (see \cite[Proposition 3.2(4)]{O1}).

\begin{cor}\label{cor:WtoW} For a very general Lagrangian space $A\subset \bigwedge^3 W$, the correspondence $k_A: H^4(\widetilde{X}_A,\mathbb Q)\to H^4(\widetilde{Y}_A,\mathbb Q)$ definied in Proposition \ref{prop:corr} induces an isomorphism of rational Hodge structures between $W(X_A)$ and $W(Y_A)$.
\end{cor}

\begin{proof} By Proposition \ref{prop:corr} $k_A$ is a map of rational Hodge structures. Hence, it is also a map of rational Hodge structures after restriction to the Hodge substructure $W(X_A)$ as described above. 
The restriction to $W(X_A)$ is non-trivial by Lemma \ref{nontrivial} and hence it is an embedding because $W(X_A)$ is an irreducible Hodge structure. Moreover, $W(X_A)$ contains $\sigma_{X_A}\wedge \sigma_{X_A}$ and   $W(Y_A)$ contains $\sigma_{Y_A}\wedge \sigma_{Y_A}$, which is a nontrivial multiple of $k_A(\sigma_{X_A}\wedge \sigma_{X_A})$ since $k_A|_{W(X_A)}$ is an injective map of rational Hodge structures. Putting these together, we obtain that $k_A$ is an isomorphism of rational Hodge structures when restricted to $W(X_A)$ and $W(Y_A)$.
\end{proof}

\begin{rem}
We expect that the correspondence is well defined not only on $W(X_A)$ but on all of $$\Sym^2 H^2 _{{\operatorname{prim}}}(\widetilde{X}_A),$$ and therefore that $k_A(\sigma_X\overline{\sigma_X})$ is proportional to $\sigma_Y\overline{\sigma_Y}$.

\end{rem}
Let us consider, as in \cite[Section 2.2]{O4}, a universal family of double EPW sextics above an open affine subset in the locus of Lagrangian spaces $\mathcal{U}'\subset LG(10,\bigwedge^3 W)$
that intersects the locus $\Sigma$.
By \cite[Proposition 3.20]{O3} we can perform simultaneously a resolution of singularity of this family on a sub-locus $\mathcal{U}\subset \mathcal{U}'$ containing
$\mathcal{U}' \cap (LG(10,\bigwedge^3W)\setminus (\Delta\cup \Sigma\cup \Gamma) )$ and a small open neighborhood of a point in $\Sigma\setminus \Sigma'$.
Here $\Sigma'$ is the sublocus of $\Sigma$ where the singularities of $\widetilde{X}_A$ or $\widetilde{Y}_A$ are worse than $A_1$. 

Note that $\mathcal U$ is affine and simply connected, we can hence choose a family of markings for $\widetilde{X}_A$, $\widetilde{Y}_A$.
Furthermore, since $\mathcal U$ is simply connected there exists a section of the local system of second cohomologies of the families $\{\widetilde{X}_A\}_{A\in \mathcal U}$ and $\{\widetilde{Y}_A\}_{A\in \mathcal U}$ associating to $A$ an element $\sigma_{X_A}\in H^2_{{\operatorname{prim}}} (\widetilde{X}_A, \mathbb C)$ representing a symplectic form on $X_A$ with the property that $\int_{X_A}\sigma_{X_A} \wedge {\sigma}_{X_A}=1$. Similarly, we define $\sigma_{Y_A}$ for $A\in \mathcal U$.
We hence have for any $A\in \mathcal{U}$ isometries 
$$\mu_{X,A}\,\colon\, H^2_{{\operatorname{prim}}} (\widetilde{X}_A, \mathbb Z)\simeq L_X$$ and $$\mu_{Y,A}\,\colon\, H^2_{{\operatorname{prim}}} (\widetilde{Y}_A, \mathbb Z)\simeq L_Y$$
that respect fixed quadratic forms on the lattices $L_X\simeq L_Y \simeq (-2)^2\oplus U^2\oplus E_8(-1)^2$ and  
$H^2_{\operatorname{prim}}(\widetilde{X}_A, \mathbb Z)$ and 
$H^2_{\operatorname{prim}}(\widetilde{Y}_A, \mathbb Z)$. 

Let us denote $\sigma^L_{X_A}:=\mu_{X,A}(\sigma_{X_A})$ the image of the class $\sigma_{X_A}$.
Similarly, $\sigma^L_{Y_A}:=\mu_{Y,A}(\sigma_{Y_A})$ denotes the image of the class $\sigma_{Y_A}$. 
Let $W_X$ and $W_Y$ be the hyperplanes in $\Sym ^2 L_X\otimes\mathbb{Q}$ and $\Sym ^2 L_Y\otimes\mathbb{Q}$  defined by the vanishing of the quadratic forms on $L_X$ and $L_Y$. Observe that $\sym^2 \sigma^L_{X_A}\in W_X$ and $\sym^2 \sigma^L_{Y_A}\in W_Y$ as for every $A$ the space $W_X$ are defined as the orthogonal complement of the class $\sigma^L_{X_A}\overline{\sigma}^L_{X_A} $ (see \cite[Section 3]{O1}) and similarly for $Y_A$ thus are isomorphic to $W(X_A)$ and $W(Y_A)$ respectively.
 
 Let us now fix $A_0$ a very general Lagrangian subspace of $\bigwedge^3 W$, then for any $A\in \mathcal U_0$ where $\mathcal U_0$ is as in Proposition \ref{prop:corr}, $k_A$ is well defined and  we have $$k_{L,A} :=(\sym^2 \mu_{Y,A})\circ k_A \circ \sym^2(\mu_{X,A}^{-1})\,\colon\, \Sym ^2 L_X\otimes\mathbb{Q} \to \Sym ^2 L_Y\otimes\mathbb{Q}.$$ The latter is a continuous family of linear maps with rational coefficients hence is constant with respect to changing $A\in \mathcal U_0$.
It follows that we have a map $\bar{k}\,\colon\, \Sym ^2 L_X\otimes\mathbb{Q} \to \Sym ^2 L_Y\otimes\mathbb{Q}$ such that $\bar k=k_{L,A}$ for $A\in \mathcal U_0$. Furthermore, as observed above  $\bar k$ maps $W_X$ to $W_Y$ i.e., it induces a linear isomorphism
$$k\,\colon\, W_X \to W_Y$$  between $\mathbb Q$ vector spaces.

\begin{prop}\label{prop:phi}
There is an isometry $\phi\,\colon\, 
L_X\otimes \mathbb{C}\to L_Y\otimes \mathbb{C}$ such that $\Sym^2\phi|_{W_X}=k$. Moreover, for all $A\in \mathcal U$, $\phi$ maps $\sigma^L_{X_A}$ to $\sigma^L_{Y_A}$ up to a constant. 
\end{prop}
The proposition will be the consequence of the following lemmas.
\begin{lem}\label{k maps sigma2 to sigma2}
The map $k$ is such that for any $A\in \mathcal U$ we have  $k(\sym^2 \sigma^L_{X_A})= c \sym^2  \sigma^L_{Y_A}$ for some constant $c\in \mathbb Q\setminus \{0\}$. 
\end{lem}
\begin{proof} By the above discussion for $A\in \mathcal U_0$ we have $k=k_{L,A}=(\sym^2 \mu_{Y,A})\circ k_A \circ \sym^2(\mu_{X,A}^{-1})$. Now,   
since by Proposition \ref{prop:corr} the maps $k_A$ for $A\in \mathcal U_0$ are maps of Hodge structures, the statement holds in the open set $\mathcal U_0$.
We conclude from the identity principle as we know that the maps $$ LG(10,\bigwedge^3V_6)\to \PP (\Sym^2 L_{X_A}\otimes \mathbb Q),\ A \mapsto [\Sym^2 \sigma_{X_A}]$$
and $$ LG(10,\bigwedge^3V_6)\to \PP (\Sym^2 L_{Y_A}\otimes \mathbb Q),\ A \mapsto [\Sym^2 \sigma_{Y_A}]$$ are meromorphic and holomorphic on the simply connected open set $\mathcal U$.
\end{proof}


\begin{lem}\label{lem:H4toH2} Let $V_1$, $V_2$ be $\mathbb{C}$ vector spaces of the same dimension. Let $H_i\subset \PP \Sym^2 V_i$ for $i\in\{1,2\}$. Let $\xi\,\colon\, H_1\to H_2$ be a linear isomorphism and $Q_i=\{v\in V_i \mid [\sym^2 v]\in H_i\}\subset V_i$. Assume that for every $x\in Q_i$ we have $\xi(\sym^2 x)= \sym^2 y$ for some $y\in Q_2$. Then there exists a linear map $\phi\,\colon\,\PP V_1\to \PP V_2$ such that $\phi(Q_1)=\phi(Q_2)$ and $(\iota \circ \xi)(t)= (\Sym^2 \phi)(t)$, where $\iota$ is the embedding 
$H_2\subset \mathbb \Sym^2 V_2$.
\end{lem}
\begin{proof} Since the Veronese embeddings $v_i\colon \PP(V_i)\to \PP \Sym^2 V_i$ given by $v_i(x)=\sym^2 x$  are isomorphisms for i=1,2, the map $\xi$ induces an isomorphism $Q_1\simeq Q_2$ which extends to a linear map $\phi\colon \mathbb P (V_1)\to \mathbb P(V_2)$ as in the assertion. 
\end{proof}

\begin{proof}[Proof of Proposition \ref{prop:phi}]
We apply Lemma \ref{lem:H4toH2} to $H_1=\PP(W_X\otimes \mathbb C)$, $H_2=\PP(W_Y\otimes \mathbb C)$,  $\xi=\PP(k)$ the map between $\PP(W_X\otimes \mathbb C) \to \PP(W_X\otimes \mathbb C)$ induced by $k$ and $Q_1$, $Q_2$ given by the quadratic forms on $L_X$ and $L_Y$. This data satisfies the assumption of Lemma \ref{lem:H4toH2} by Lemma \ref{k maps sigma2 to sigma2}. We deduce the existence of a linear map $\overline{\phi}: \PP(L_X\otimes \mathbb C)\to \PP(L_Y\otimes \mathbb C)$ such that $\sym^2 \bar{\phi}=\PP(k)$ and in particular $\overline{\phi}([\sigma^L_{X_A}])= [\sigma^L_{Y_A}]$. Furthermore, since $\bar{\phi}(Q_1)= Q_2$ we can choose $\phi\colon  L_X\otimes \mathbb C\to L_Y\otimes \mathbb C$ such that $\phi$ is an isometry and $\overline{\phi}$ is the map of projective spaces induced by $\phi$.
\end{proof}

Let us hence fix a linear isometry
$$\phi\colon L_{X}\otimes \mathbb C\mapsto L_{Y}\otimes \mathbb C,$$
as in Proposition \ref{prop:phi}. Then 
for any $A\in \mathcal U$ the map 
\begin{equation}\label{phiA}
\phi_A:=\mu_{Y,A}^{-1}\circ \phi\circ \mu_{X,A}: H^2_{prim}(X_A,\mathbb C)\to H^2_{prim}(Y_A,\mathbb C)    
\end{equation}
 is an isomorphism of polarised complex Hodge structures.

It remains to prove that $\phi$ maps $L_X$ to $L_Y$ which would imply directly that $\mu_{Y,A}^{-1}\circ \phi\circ \mu_{X,A}$ is an isometry of polarised integral Hodge structures. 
Let us consider a very general Lagrangian space  $A\in \Sigma\cap \mathcal U$. In that case $\PP(A)\cap G(3,6)$ is a point.
As observed in \cite[Section 3]{O3} and \cite[Section 3.1]{IKKR1}, the corresponding hyper-K\"ahler manifolds $\widetilde{X}_A$ and $\widetilde{Y}_A$ both admit a 
birational contraction. In fact, $\operatorname{Pic}(\widetilde{X}_A)=(2)\oplus  (-2)$ and $\operatorname{Pic}(\widetilde{Y}_A)=(4)\oplus (-2)$. Denote by $\widetilde{Z}_A$ the base of the 
contraction on $\widetilde{Y}_A$, which is a double EPW quartic associated to $A$.

\begin{lem}\label{lem:moduli_isometry} The fourfold $\widetilde{X}_A$ is isomorphic to the moduli space of twisted sheaves $M_v(S_2,\beta)$, where $S_2$ is a K3 surface of degree $2$ and $\beta \in H^2(S_2,\oo_{S_2}^{\ast})$ is a $2$-torsion Brauer class. Moreover, we can choose $S_2$ and $\beta$ such that $\widetilde{Y}_A$ is isomorphic to $M_w(S_2,\beta)$ with the same K3 surface and Brauer class. In particular, we have an isometry $\varphi$ of integral Hodge structures $T(\widetilde{X}_A)$ and $T(\widetilde{Y}_A)$ which naturally extends to an isometry between $H^2_{{\operatorname{prim}}}(\widetilde{X}_A,\mathbb{Z})$ and $H^2_{{\operatorname{prim}}}(\widetilde{Y}_A,\mathbb{Z})$ .
 \end{lem}
 \begin{proof}
 We first prove that $T(\widetilde{Y}_A)$ and $T(\widetilde{Z}_A)$ are Hodge isometric, then that $T(\widetilde{X}_A)$ and $T(\widetilde{Z}_A)$ are isomorphic.
 The first follows from the fact that from \cite[Section 4.2]{vGK} the sixfold $\widetilde{Y}_A$ and the base of the contraction are moduli spaces of twisted sheaves on the same K3 surface with the same Brauer class (cf.~\cite{BM}).
 
 Now let us prove the second isomorphism.
  From \cite[Section 5.1]{vGK}, we deduce that $\widetilde{X}_A$ is isomorphic to the moduli space of twisted sheaves on the K3 surface $S_2$ associated with $A$ and the Brauer class $\beta$ with $B$-lift $B\in \frac{1}{2}H^2(S_2,\Z)$ such that $Bh=0$, $B^2=\frac{1}{2}$ (induced by the associated Verra threefold) and Mukai vector $v=(2,2B,0)$.
 It follows from \cite[Section 5]{CKKM}, \cite{IKKR2} that the EPW quartic which is the base of the contraction on $\widetilde{Y}_A$ is the moduli space of twisted sheaves on the same surface $S_2$ with the same Brauer class $\beta$ but Mukai vector $v=(0,h,0)$.
 
 Finally, let us prove that $\varphi$ extends to $H^2_{{\operatorname{prim}}}(\widetilde{X}_A,\mathbb{Z})$. To do so, it suffices to define a natural integral isometry on the remaining generator, which is the effective divisor of square $-2$, which we send to the effective divisor of square $-2$ inside $H^2_{{\operatorname{prim}}}(\widetilde{Y}_A,\mathbb{Z})$. 
 \end{proof}

Notice that in the above proof we prove that $\widetilde{Z}_A$ and $\widetilde{X}_A$ are associated to the same twisted K3 surface $(S_2,B)$ and $\widetilde{Z}_A$ and $\widetilde{Y}_A$ are also associated with the same twisted K3 surface $(S_2',B')$, but we did not prove $(S_2,B)=(S_2',B')$.

Next, let us prove that $\varphi$ is, up to a constant, the above identification.

\begin{lemma}\label{lem:integral_on_divisor}
Assume $A$ is very general in $\mathcal U\cap \Sigma$ and let $$\phi_A\,\colon\, H^2_{{\operatorname{prim}}}(\widetilde{X}_A,\mathbb{C}) \to H^2_{{\operatorname{prim}}}(\widetilde{Y}_A,\mathbb{C}) $$ be the isometry induced by $\phi$. Then $\phi_A$ coincides with the isometry $\varphi$ on $T(\widetilde{X}_A)\otimes\mathbb{C}$ defined in Lemma \ref{lem:moduli_isometry}.
\end{lemma}
\begin{proof}
Let us consider the Hodge morphism $\phi_A-\varphi$ on $T(\widetilde{X}_A)\otimes\mathbb{C}$. Since both morphisms map the line $\langle\sigma_X\rangle$ to the line $\langle\sigma_Y\rangle$ and they are isometries, we have $(\phi_A\pm\varphi)(\sigma_X)=0$ for a choice of sign $\pm$. Hence this morphism is zero in the smallest Hodge structure containing $\sigma_X$, which is precisely $T(\widetilde{X}_A)$, therefore the two morphisms coincide up to a sign. This sign choice is locally constant, so if we obtain $\phi_A=-\varphi$, we can replace $\phi$ with $-\phi$ in Proposition \ref{prop:phi}, so that the claim follows.
\end{proof}

\begin{lemma}\label{lem:H2integers}
Let $A \in \mathcal U$ and $\phi$ be an isometry as in Proposition \ref{prop:phi}. Then the Hodge isometry $$\phi_A\colon H^2_{{\operatorname{prim}}}(\widetilde{X}_A,\mathbb{C}) \to H^2_{{\operatorname{prim}}}(\widetilde{Y}_A,\mathbb{C})$$ defined by (\ref{phiA}) restricts to an isomorphism 
$$H^2_{{\operatorname{prim}}}(\widetilde{X}_A,\mathbb{Z})\to H^2_{{\operatorname{prim}}}(\widetilde{Y}_A,\mathbb{Z}).$$
\end{lemma}
\begin{proof}
By Lemma \ref{lem:integral_on_divisor} and Lemma \ref{lem:moduli_isometry}, we know that $\phi$ induces an integral isometry between the lattices $\mu_{X,A'} (T(\widetilde X_A'))$ and $\mu_{Y,A'}^{-1}(T(\widetilde Y_A'))$  for a fixed $A'\in \mathcal U\cap \Sigma$. The two lattices are primitive $21$-dimensional sublattices of $L_X$ and $L_Y$. The orthogonal complements of these sublattices are generated each by a single integral class, which we call $D_X$ and $D_Y$ respectively. As $\phi$ is an isometry and $D_X^2=D_Y^2=-2$, we must have either $\phi(D_X)=D_Y$ or $\phi(D_X)=-D_Y$. So either $\phi$ coincides with the natural isometry $L_X\cong L_Y$ induced by $\varphi$ of Lemma \ref{lem:moduli_isometry}, or it coincides up to the action of the reflection with center $D_Y$. As this reflection is integral, in both cases $\phi$ is integral as well. In consequence, $\phi_A$ is also integral. 
\end{proof}

\begin{proof1} Let us take $A$ very general in $\mathcal U$. By Propoosition \ref{prop:phi} there exists an isometry $\phi:L_X\otimes \mathbb C\to L_Y \otimes \mathbb C$  such that $\phi_A$ is a complex Hodge isometry. Furthermore, by Lemma \ref{lem:H2integers} the map $\phi_A$ is integral and in consequence is an integral Hodge isometry.
\end{proof1}

As a direct consequence of Theorem 1.1, we obtain Corollary \ref{cor:stability}.

\begin{proof2}
By \cite[Theorem 1.7]{PPZ}, for a very general Lagrangian $A$ there exists a stability condition $\sigma$ on $\operatorname{Ku}(A)$ and a Mukai vector $v=\lambda_1+\lambda_2$ such that $M_\sigma(\operatorname{Ku}(A),v)$ is smooth. Here, $\lambda_1$ and $\lambda_2$ are two orthogonal generators of the integral part of the Hochschild cohomology of $\operatorname{Ku}(A)$ of square 2. By \cite[Theorem 1.5]{PPZ}, we have $v^\perp\cong H^2(M_\sigma(\operatorname{Ku}(A),v),\ZZ)$. The transcendental part of the above Hodge structure is given by $(\lambda_1-\lambda_2)^\perp \cap (\lambda_1+\lambda_2)^\perp$, as by our assumption on $A$, $\lambda_1$ and $\lambda_2$ are the only algebraic integral classes. Thus, 
\begin{align*}
    H^2_{\operatorname{prim}}(M_\sigma(\operatorname{Ku}(A),v)) &=(\lambda_1-\lambda_2)^\perp \cap (\lambda_1+\lambda_2)^\perp =\\
    =\lambda_1^\perp \cap \lambda_2^\perp&=H^2_{\operatorname{prim}}(M_\sigma(\operatorname{Ku}(A),\lambda_1))=H^2_{\operatorname{prim}}(M_\sigma(\operatorname{Ku}(A),\lambda_2))
\end{align*}
where $M_\sigma(\operatorname{Ku}(A),\lambda_1)$ and $M_\sigma(\operatorname{Ku}(A),\lambda_2)$ are isomorphic either to $\widetilde{X}_A$ or $\widetilde{X}_{A^{\vee}}$ by \cite[Proposition 5.17]{PPZ}. Therefore, by Theorem \ref{thm:main}, the double EPW cube $\widetilde{Y}_A$ is Hodge isometric to $M_\sigma(\operatorname{Ku}(A),v)$, hence isomorphic to it by the global Torelli theorem \cite{V1}. 
\end{proof2}


\begin{rem}
In a recent preprint \cite[Section 5.2]{GLZ}, Guo, Liu and Zhang prove that the moduli space of stable objects on a very general Gushel--Mukai fourfold $X$ with Mukai vector $\lambda_1+\lambda_2$ contains naturally the maximal rational quotients of the moduli spaces of twisted cubics on a Gushel--Mukai threefold contained in the fourfold.  Together with Theorem \ref{thm:main}, this provides strong evidence for the fact that the (at least very general) EPW cube should correspond to the maximal rational quotient of the moduli space of twisted cubics on a Gushel--Mukai fourfold.

\end{rem}





\section{Period domains and period maps}\label{s7}

\subsection{Period maps}
To a Lagrangian $A \subset \bigwedge^3 W$, we can associate the corresponding point in the period domains associated to $\widetilde{X}_A$ and $\widetilde{Y}_A$.

More precisely, by associating to $A\in LG(10,\bigwedge^3 W)$ the Hodge structure on $\mathrm{H}^2(\widetilde{X}_A,\mathbb{Z})=:{\Lambda_X}$ or $\mathrm{H}^2(\widetilde{Y}_A,\mathbb{Z})=:{\Lambda_Y}$ we obtain 
 two rational period maps:
$$\mathcal{P}_X\,\colon\, LG(10,\textstyle \bigwedge^3W) \dashrightarrow \mathbb{D}_{X},$$

$$\mathcal{P}_Y \,\colon\, LG(10,\textstyle \bigwedge^3W)\dashrightarrow \mathbb{D}_{Y}.$$

 Here, $\mathbb{D}_{X}$ and $\mathbb{D}_{Y}$ are the quasi-projective period domains of double EPW sextics i.e.~manifolds of K3$^{[2]}$ type with a polarization of degree $2$ and divisibility 1, and double EPW cubes i.e. manifolds of K3$^{[3]}$ type with a polarization of degree $4$ and divisibility 2 respectively. It follows from Eichler criterion \cite[Lemma 3.5]{GHS} that both polarizations are unique up to isometry in the respective lattices $\Lambda_X$ and $\Lambda_Y$. 
 The orthogonal complements of those polarizations are in both cases isomorphic to 
\begin{equation}\label{eq:lattice}
\Lambda := U^2 \oplus E_8(-1)^2 \oplus (-2)^2.
\end{equation}
The two period domains $\mathbb{D}_{X}$ and $\mathbb{D}_{Y}$ are obtained as the quotient of $\Lambda$ by the respective groups ${O}_X(\Lambda)$ and $O_Y(\Lambda)$ of automorphisms of $\Lambda_X$ and $\Lambda_Y$ preserving the polarization. 
There is a degree two covering map 
\begin{equation}\label{eq:rho}
\rho \,\colon\,  \mathbb{D}_{X}\to \mathbb{D}_{Y}
\end{equation}
by \cite[Proposition 1.2.3, item (4) and Section 2.3]{LO}: in the notation of that paper, we have $\mathbb{D}_X=\mathcal{F}_\Lambda(\widetilde{O}^+(\Lambda))$ and $\mathbb{D}_Y=\mathcal{F}(O^+(\Lambda))=\mathcal{F}(20)$. More precisely, the covering map is induced by the index two inclusion $\widetilde{O}^+(\Lambda)\subset O^+(\Lambda)$, whose complement is represented by the isometry $\iota$ of $\Lambda$, defined by exchanging the two generators of degree $-2$, which lies in $O_Y(\Lambda)$ but not in $O_X(\Lambda)$.

As a consequence of Theorem \ref{thm:main}, we can compare the period maps $\mathcal{P}_X$ and $\mathcal{P}_Y$ by means of $\rho$.

  \begin{cor}\label{periods}
The composed map
$$\textstyle LG(10,\bigwedge^3 W)\stackrel{\mathcal{P}_X}{\dashrightarrow}
 \mathbb{D}_{X}\xrightarrow{\rho} \mathbb{D}_{Y},$$ 
is $\mathcal P_Y$ and it is well defined for all $[A]\in LG(10,\bigwedge^3 W)$ such that $\PP(A)\cap G(3,W)=\emptyset$.
\end{cor}
\begin{proof}
    Let $A$ be a Lagrangian as in \eqref{eq:Lagrangian_condition}. The maps $\mathcal{P}_X$ and $\mathcal{P}_Y$ are well defined as both $\widetilde{X}_A$ and $\widetilde{Y}_A$ are smooth.
    Following \cite[Section 2.3]{LO}, we see that the covering involution defined by $\rho$ has a natural geometric interpretation. Indeed, the involution induces an involution on the space $\textstyle LG(10,\bigwedge^3 W)$ which sends a Lagrangian to its dual. By Theorem \ref{thm:main}, for such a Lagrangian $A$ we have that $\mathcal{P}_Y=\rho\circ \mathcal{P}_X$, as $\widetilde{X}_A$ and $\widetilde{Y}_A$ have the same $H^2_{prim}$.
    
    The map $\rho$ is everywhere defined, and by \cite[Theorem 0.3]{O3}, the map $\mathcal{P}_X$ is well defined for all Lagrangians such that $\PP(A)\cap G(3,W)=\emptyset$. As $\mathcal{P}_Y=\rho \circ \mathcal{P}_X$, we conclude that $\mathcal{P}_Y$ is well defined on the same set of Lagrangians as $\mathcal{P}_X$. 
\end{proof}
\begin{rem}
Actually, by \cite[Theorem 0.2]{O3}, the map $\mathcal{P}_X$ is defined also on an open set of Lagrangians such that $\PP(A)\cap G(3,W)\neq \emptyset$, which are the Lagrangians that we used in Lemma \ref{lem:integral_on_divisor}. However, we do not need a precise description of this set of Lagrangians in the following, and we decided to stick to the simple generality assumption $\PP(A)\cap G(3,W)=\emptyset$
in the above corollary.

\end{rem}

\subsection{The image of the period maps}



By analogy with the paper \cite{O3}, we study the image of the period maps $\mathcal{P}_X$ and $\mathcal{P}_Y$. We describe some Heegner divisors inside $\mathbb{D}_X$ which lie (partially) outside the image of those period maps (cf.~\cite{B1} for the case of Debarre-Voisin varieties).

Denote by $\mathcal{M}$ the GIT  quotient
$$LG(10,\bigwedge^3 W)//\mathrm{PGL(W)} $$
parametrizing Lagrangians up to isomorphism. 
Consider its natural subset $\mathcal N\subset \mathcal M$ which parametrizes Lagrangian spaces with a decomposable tensor.
The period map descends to a rational map
$$p \,\colon\, \mathcal{M}\dashrightarrow \mathbb{D}_X^{BB}$$ where $\mathbb{D}_X^{BB}$ is the Baily-Borel compactification of the period domain.
In \cite[Theorem 0.2]{O3} O'Grady considered four divisors  $\mathbb{S}^{\ast}_2$, $\mathbb{S}_2'$, $\mathbb{S}_2''$, $\mathbb{S}_{4}$
 in the period domain that are in the complement of the image of the map $p|_{\mathcal{M}\setminus\mathcal{N}}$ (from \cite[Theorem 0.2]{O3} the map $p$ is regular and injective on $\mathcal{M}\setminus\mathcal{N}$ and its image is contained in $\mathbb{D}_X$). 
Moreover, the following descriptions of elements that fill out open dense subsets of those divisors was given by O'Grady in the same paper.

\begin{enumerate}
\item
 The divisor $\mathbb{S}_{4}$ parametrizes Hilbert schemes of two points on a quartic K3 surface: these have a natural involution (described by Beauville in \cite[Section 6]{beau_trivial})
 and a $6:1$ map to $G(2,4)$, which factors through the quotient by the above involution. In a sense, $G(2,4)$ with a triple structure can be seen as an EPW sextic, but the polarized manifold $(S^{[2]},H-\delta)$ does not correspond to a unique Lagrangian. Here, $H$ is the nef divisor induced on $S^{[2]}$ from the ample divisor of square 4 on $S$, and $2\delta$ is the class of the exceptional divisor of non-reduced subschemes. 
\item
Points in the two divisors $\mathbb{S}_2'$, $\mathbb{S}_2''$ correspond respectively to the Hilbert scheme of two points on a degree two K3 surface and the moduli space of stable sheaves on a K3 surface $S$, such that $S$ has a polarization $H$ of degree two and the Mukai vector is $(0,H,0)$. 
\item Finally, points in the divisor $\mathbb{S}^{\ast}_2$ correspond to EPW sextics with at least a singular plane, and their periods are related to the period of the K3 surface obtained as double cover of this plane branched along a sextic (plus a $2$-torsion Brauer class \cite{vGK}).
\end{enumerate}


By analogy to the case of double EPW sextics, we consider in the period domain $\mathbb{D}_Y$ of double EPW cubes the set of Hodge structures perpendicular to a $(1, 1)$-root 
of negative square. These are the images through the covering map $\rho\,\colon\,\mathbb{D}_X\to \mathbb{D}_Y$ of the above divisors by Corollary \ref{periods} and \cite[Theorem 0.2]{O3}. So from Corollary \ref{periods}, the period map $\mathcal M\setminus\mathcal N\to \mathbb{D}_Y$ is then regular and has image whose complement is the sum of the following divisors:
\begin{enumerate}
\item[$\mathbb{S}_4$:] $D^2 = -4$ and $\operatorname{div}(D) = 2$, 
\item[$\mathbb{S}^{\ast}_2$:] $D^2 = -2$ and $\operatorname{div}(D) = 1$, first orbit (EPW quartic),
\item[$\mathbb{S}^{\#}_2$:] $D^2 = -2$ and $\operatorname{div}(D) = 1$, second orbit.

\end{enumerate}
In the above list, we specified the $(1,1)$ root $D$ orthogonal to the period.
These three elements are in the orbit for the action of $O_Y(\Lambda)$ of $t_1 + t_2$, $d$, $t_1$, where $t_1$ and $t_2$ are generators of the orthogonal summands of square $-2$ in \eqref{eq:lattice} and $d\in (t_1\oplus t_2)^\perp$ has square $-2$. These are the only orbits of elements of this square and divisibility by Eichler's criterion \cite[Lemma 3.5]{GHS}.
Notice that the two classes $d$ and $t_1$, although they both have square -2 and divisibility one once they are embedded in $H^2(Y,\ZZ)$, have divisibility respectively $1$ and $2$ in $\Lambda$, hence they belong to different $O_Y(\Lambda)$ orbits.


Let us analyze in detail the hyper-K\"ahler sixfolds corresponding to general points in the above divisors and describe using Corollary \ref{periods} the associated EPW cubes (i.e.~Lagrangian degeneracy loci in $G(3,6)$ associated to the corresponding Lagrangian subspace).
\subsection{The case $\mathbb{S}_4$} Let $X \to \mathbb{P}^3$ be a generic manifold of $K3^{[3]}$-type having a Lagrangian fibration with a section. It was proven in \cite{MW}  that $\operatorname{Pic}(X) = U(2)$ and there are elements of divisibility $2$ in it. Let $F$ be the class of a pull back of a plane and let $E$ be the dual class to  twice the class of a line in the section. From \cite[Theorem 1.1]{hass_harv_tsch} we know that $E^2 = -12$ and $\operatorname{div}(E) = 2$. The class $E+4F$ has square $4$ and divisibility $2$ and it is orthogonal to $E-2F$, which has square $-4$ and divisibility $2$. The class $E+ 4F$ is movable (as $E$ and $F$ are effective) but not nef on X, as it pairs negatively with $E$. Therefore, there exists a hyper-K\"{a}hler manifold $Y$ and a birational map $f \,\colon\, Y \dashrightarrow X$, such that $f^{\ast}(E + 4F)$ is nef. Indeed, $Y$ is the Mukai flop of the section as we know that the movable cone is divided in two parts. Let us show that $Y$ can be also seen as a moduli space of sheaves on a K3 surface of degree $4$. 
\begin{prop}
Let $(S,H)$ be a K3 surface of degree four and let us consider the polarized fourfold $(S^{[2]},H-\delta)$, whose period is in $\mathbb{D}_X$.
 Then its image in $\mathbb{S}_4\subset
 \mathbb{D}_Y$
 through the map $\rho$ of \eqref{eq:rho} corresponds to a pair $(M_{v}(S),(2,-H,0))$, where is a moduli space of stable objects on $S$ with Mukai vector $v=(0,H,-2)$ and $(2,-H,0)$ a polarization of BBF degree $4$.

\end{prop}
In the above proposition, $(2,-H,0)\in H^*(S,\ZZ)$ is an element of the extended K3 lattice seen as the class of a divisor on $M_{v}(S)$ through Mukai's isomorphism $v^\perp\cong H^2(M_v(S),\ZZ)$ (cf. \cite[Appendix B]{D}).
\begin{proof}
By Corollary \ref{periods}, the period of $(S^{[2]},H-\delta)$ and its image in $\mathbb{D}_Y$ correspond to manifolds with isometric primitive cohomology and, in particular, with isometric transcendental lattice. A moduli space of stable objects on $S$ inherits the same transcendental lattice of $S$, thus corresponds to the given period point in $\mathbb{D}_Y$ (see \cite{MW}). The classes $H$ and $\delta$ of $S^{[2]}$ can be obtained inside two copies of $U$ contained in the Mukai lattice of $S$, whose standard generators we call $e_1,f_1$ and $e_2,f_2$. We have $e_1+f_1=H-\delta$, $e_2+f_2=(1,0,-1)$ (the Mukai vector corresponding to the moduli space $S^{[2]}$) and $e_1-f_1+e_2-f_2=H-2\delta$. In order to describe the considered double EPW cube, we have to choose $v=e_1+e_2+f_1+f_2$ and $H'=e_1+f_1-e_2-f_2$, and the result follows from a direct computation.  
\end{proof}
In the above notice that the class $(2,-H,0)$ is not ample on $M_{v}(S)$ and the class $(0,0,1)$ is not nef, so that the natural Lagrangian fibration is not everywhere defined and we need to change the biratonal model.

\begin{rem}
From Corollary \ref{periods} we infer that the 
 Lagrangian space $A_1$ associated to the corresponding EPW sextic cuts $G(3,6)$ along the second Veronese of $\PP^3$ and the EPW sextic is the triple quadric. We can show that the EPW cube corresponding to $A_1$ is an eightfold in $G(3,6)$, which is the set of planes cutting a given $4$-dimensional quadric in two lines.
This eightfold is singular in a sixfold $T$ of degree $64$ consisting of planes cutting the quadric in a double line.
The sixfold contains two second Veronese embedding of $\PP^3$; $\mathbb V_1=\PP(A_1)\cap G(3,6)$ and $\mathbb V_2=\PP(A_1^{\vee})\cap G(3,6)$ it can be seen as the sum of threefold cones with vertex $v\in \mathbb V_1$ and base the second Veronese surface corresponding to the dual plane to $v$ in $\mathbb V_2$. One can show that $T$ is contained in the locus $D^3_{A_1,\mathcal T}$, but $\mathbb V_1$ in $D^4_{A_1,\mathcal T}$.
We expect that the images of fibers of the Lagrangian fibration on $M_v(S)$ intersect
in Veronese surfaces. 
\end{rem}

\subsection{The cases $\mathbb{S}_2^{\ast}$ and $\mathbb{S}^{\#}_2$.}
The $\mathbb{S}^{\ast}_2$ divisor is described analogously to the EPW sextic case: it is given by Lagrangians containing a decomposable tensor and its general elements were described fully in \cite{IKKR2}. \\
The $\mathbb{S}^{\#}_2$ case corresponds to the image of the two divisors $\mathbb{S}'_2$ and $\mathbb{S}_2''$, and does not correspond to stable Lagrangians: in this case, we can describe the periods as follows. If we take a degree $2$ K3 surface $(S,h)$ and its Hilbert scheme $S^{[3]}$, we have two natural classes on $S^{[3]}$: the nef class $H$ induced by the polarisation $h$ of the K3 surface $S$ and the class $\delta$ such that $2\delta$ is the class of the locus of non-reduced subschemes. The class $2H-\delta$ has square $4$ and divisibility $2$. 
It is orthogonal to the class $H-\delta$, which corresponds to an effective divisor (as it has degree $-2$ class), which is hence contracted by the map 
\begin{equation}\label{eq:phi}
\varphi:=\varphi_{|2H-\delta|}\,\colon\, S^{[3]}\dashrightarrow \mathbb{P}^{19}
\end{equation}
associated to the linear system $|2H-\delta|$.
As $\delta$ is of divisibility $4$ we find that $12H-9\delta$ is a class of square $-36$ and divisibility $4$. It divides the movable cone in two K\"ahler chambers of birational models (see \cite{Mon}). 
So we can find a birational model of $S^{[3]}$ where $2H-\delta$ is ample.
In order to describe the image of $\varphi$ recall that the planes tangent to a fixed Veronese surface $v_2(\PP^2)$ in $\PP^5=\PP(W)$ can be described as a third Veronese surface  $v_3(\PP^2)\subset G(3,W)\subset \PP^{19}$. The span of $v_3(\PP^2)\subset \PP(\bigwedge^3 W)$ is a $9$-dimensional projective space $\PP(A_0)$ such that $A_0$ is Lagrangian for the symplectic structure on $\bigwedge^3 W$.
\begin{lem}
The map $\varphi$ as in \eqref{eq:phi} is an $8:1$ map to a subset of $G(3,W)$ that can be described as the set of planes intersecting the fixed Veronese surface $v_2(\PP^2)\subset \PP(W)$ in three points.
The image of $\varphi$ is contained in the EPW cube constructed from $A_0$.

\end{lem}
\begin{proof}
\textcolor{black}{First, the system $|2h|$ is still $2:1$ on $S$. So $|2H|$ factorises through $(\PP^2)^{[3]}$ because a general element from $|2H|$ can be described as a triple $p+q+r$ such that $p$ is contained in a fixed divisor in $|2h|$ (moreover $H^0(\oo_{S^{[3]}}(2H))=\sym^3H^0(\oo_S(2h)$ as they have the same dimension). The system $|2H-\delta|$ is a subsystem of $|2H|$ and thus} on an open subset of $S^{[3]}$ (of triples mapping to distinct points on $\PP^2$ through $|h|$), the map $\varphi$ factorizes through the birational map $f\,\colon\,(\PP^2)^{[3]}\dashrightarrow G(3,6)$
that associates to three points on the Veronese surface in $\PP^5$ the plane spanned by those points. 
We consider the divisor of non reduced points $D\subset(\PP^2)^{[3]}$ and the divisor $K \subset (\PP^2)^{[3]}$ of triples of points with one on a given line in $\PP^2$.
Then $|2K-D/2|$ gives
the morphism $f$ because and from \cite[p.~91]{EB} they both contract
the same curve $\mathcal{C}_{00}$ of triples
such that two points are fixed and the third one is on the line determined by them (and the Picard rank of 
$(\PP^2)^{[3]}$ is $2$). 
Thus the map $S\to \PP^2$ given by $|h|$
induces a generically $8:1$ rational map
$S^{[3]}\dashrightarrow  (\PP^2)^{[3]}$ such that the pullback of $2K-D/2$ is $2H-\delta$.

We claim that the image of $\varphi$ is contained in the EPW
cube corresponding to the considered Lagrangian $A_0$.
Indeed, consider (looking at the dual picture) the set of planes disjoint from the Veronese surface $V\subset \PP^5$ that intersect the secant cubic to the Veronese in three lines. 
Observe that for a fixed plane $U_1\subset G(3,6)$ such lines are contained in three different planes $T_1,T_2,T_3$ tangent to $V$. 
Such planes give points on the EPW cube as the tangent planes to $G(3,6)$ at those points cuts the third Veronese contained in $G(3,6)\subset \PP^{19}$ (of tangent spaces to $V$) in three points.
Indeed, it is enough to recall that the tangent space at a point $U\subset G(3,6)$
is the set of planes cutting $U$ in a line. Thus the points on $V$ corresponding to $T_1,T_2$ and $T_3$ are in the intersection of $T_{U_1}$ and $\PP(A_0)$, so $[U_1]\in D^3_{A_0}$.
\end{proof}
\begin{rem}
Note that an EPW cube has degree $480$, so the image of $\varphi$ cannot be equal to it as it has degree $57$  (from \cite[p.~110]{EB} we have $(2K-D/2)^4=57$). Moreover, one can prove using Macaulay2 that the EPW cube corresponding to $A_0$ has dimension $6$ (but we dont know it is reduced). The image of $\varphi$ is thus one of its components.
\end{rem}

This is a $\PP^2$ fibration over $S$.These base points cannot be eliminated by flops.


\begin{rem}
It is natural to consider in the period domain divisors corresponding to singular EPW cubes. We have at least two more such divisors.
This is $\mathbb{S}^{\ast}_{10}\subset \mathbb{D}_X$ such that $D^2 = -10$ and $\operatorname{div}(D) = 2$
and the divisor $\mathbb{S}^{\ast}_{12}$ such that $D^2 = -12$ and $\operatorname{div}(D) = 2$ in $\mathbb{D}_Y$ (the class is the monodromy orbit of $2d + t_1 + t_2$).

  For $\mathbb{S}^{\ast}_{10}\subset \mathbb{D}_Y$ the corresponding Lagrangian $A$ as in \cite[Section 4]{IKKR1} is general among Lagrangians for which the degeneracy
locus $D^3_{A, \mathcal F}\neq \emptyset$ (see Section \ref{sec:prel}(1)). By \cite[Section 4]{IKKR1}, its corresponding double EPW cube is a hyper-K\"ahler manifold $\widetilde{Y}_A$ birational to $ S^{[3]}$, where $S$ is a K3 surface of degree $10$. In this case, $\widetilde{X}_A$ is a singular variety which admits a small resolution $\overline{X}_A$ which is a hyper-K\"ahler fourfold such that $\overline{X}_A\cong S^{[2]}$. Here, the correspondence $k_A$ defined as in Proposition \ref{prop:corr} is still a well defined non-trivial correspondence between Hodge structures of $\overline{X}_A$ and $\widetilde{Y}_A$.

Finally from \cite[Remark 5.29]{DK period} and Theorem \ref{thm:main} the divisor $\mathbb{S}^{\ast}_{12}$ corresponds to those Lagrangians for which $D^4_{A,\mathcal{T}}$ is not empty.

\end{rem}




\section*{Appendix}
We provide here a program describing explicitly a general EPW sextic and a line four-secant to its singular locus. We do the computations in positive characteristic. For a fixed line $l$ in $\mathcal Z$ we choose a general Lagrangian space containing it and consider the corresponding EPW sextic. We then check that the corresponding line $L$ in $\mathbb{P}^5$ is four-secant to the singular locus $S$ of this EPW sextic.  Since the reduction to positive characteristic of $L\cap S$ consists of four distinct points then these are also four distinct points over $\mathbb{C}$. 
\begin{verbatim}
R=ZZ/107[x,y,z,t,u,v]
M=matrix{
    { 0,0,0,0,0,0,0,u,-t,z},
    {0,0,0,0,0,-u,t,0,0,-y},
    { 0,0,0,0,u,0,-z,0,y,0},
    {0,0,0,0,-t,z,0,-y,0,0},
    {  0,0,0,0,0,0,0,0,0,x},
    { 0,0,0,0,0,0,0,0,-x,0},
    {  0,0,0,0,0,0,0,x,0,0},
    {  0,0,0,0,0,0,0,0,0,0},
    {  0,0,0,0,0,0,0,0,0,0},
    {  0,0,0,0,0,0,0,0,0,0}
    }
MM=M+transpose M
K=ZZ/107[wa_1..wa_55,Z]
A=genericSymmetricMatrix(K,10)
SY=gens kernel transpose ((coefficients(flatten((
A*transpose(matrix{{0,-1,1,0,0,0,1,0,0,0},{0,0,1,0,-1,0,1,0,0,0}}))-
transpose matrix{{0,0,0,0,0,0,0,Z,0,0},{0,0,0,0,0,0,0,0,Z,0}}),
Monomials=>vars K))_1)
RT=transpose(SY*random(K^37,K^1)) 
-- we choose in a random way entries of A satisfying the assumptions 
NMM=sub(sub(A, RT),R);
NMM=sub(1/((NMM*transpose(matrix{{0,-1,1,0,0,0,1,0,0,0},
{0,0,1,0,-1,0,1,0,0,0}}))_(7,0)), ZZ/107)*NMM;
NMM-transpose NMM;
NMM-transpose NMM
EPW=(decompose (ideal det(NMM-MM)))_0 -- the EPW sextic
degree EPW 
dim EPW
S=(ideal singularLocus EPW) -- the singular surface of the EPW sextic
dim S
S=saturate S;
degree S
L=ideal(z,t,u,y+x-1) -- the line that we check to be four-secant to S
L:EPW
dim(L+S)
degree radical(L+S)-- L is four-secant to S
decompose(L+S)

\end{verbatim}

The second part of the program checks on an example that the four lines obtained from the projection from the line $L$ of the tangent spaces to $S$ in the four points of intersection $L\cap S$ do not lie on a quadric surface. Here, for simplicity we fix $RT$ in the program above and extend the field in which the varieties are defined, so that the four point in $L\cap S$ are defined over the field. Again, observe that the computation in positive characteristic implies that in the general setup over $\mathbb C$ the generality result on the position of the four lines in $\mathbb P^3$ also holds.
Below we assume that the script above was run with 

\begin{verbatim}
RT=sub(matrix{{-42,50,-23,-36,50,39,-34,40,-14,-21,38,-22,5,38,43,-47,
-38,-9,8,31,1,-22,4,-53,-12,-52,-33,51,5,-40,4,24,49,-6,38,43,-47,-17,
-30,8,-26,39,-53,-45,53,6,-5,43,41,46,-42,40,14,15,15,21}},K)
\end{verbatim}
We continue with the output from the previous script.
\begin{verbatim}
U=ZZ/107[a]/ideal(a^2-7*a+52) -- we define the extension field
AS=U[x,y,z,t,u,v]
LL=(map(AS,R))L -- we consider L over the new field
LS=(map(AS,R))S -- we consider L over the new field
decompose(LL+LS) -- now the intersection decomposes into four points
--defined over the field
P1=ideal mingens ideal(matrix{{x+33,y-34,z,t,u,v-1}}*
sub(jacobian LS, {u=>0,t=>0,z=>0,y=>34,x=>-33}))
P2=ideal mingens ideal(matrix{{x+48,y-49,z,t,u,v-1}}*
sub(jacobian LS, {u=>0,t=>0,z=>0,y=>49,x=>-48}))
P3=ideal mingens ideal(matrix{{x+a-1,y-a,z,t,u,v-1}}*
sub(jacobian LS, {u=>0,t=>0,z=>0,y=>a,x=>1-a}))
P4=ideal mingens ideal(matrix{{x-a+6,y+a-7,z,t,u,v-1}}*
sub(jacobian LS, {u=>0,t=>0,z=>0,y=>7-a,x=>a-6}))
--P1, P2,P3, P4 are the tangent spaces in the four points
PRL=U[e_1,e_2,e_3,e_4]
J=ideal mingens preimage(map(AS,PRL, gens LL), intersect(P1,P2,P3,P4))
--the projection of the union of P1,P2,P3,P4 from the line LL
betti J -- we check that there is no quadric in the ideal 

\end{verbatim}

\end{document}